\documentclass[11pt,reqno]{amsart}
\usepackage{amsmath,amssymb}
\usepackage{amsfonts}
\usepackage{xcolor}
\usepackage{mathtools}
\usepackage{nccmath}
\usepackage{amsthm}
\usepackage[utf8]{inputenc}
\usepackage{xcolor}
\usepackage{hyperref}
\hypersetup{
    colorlinks=true,
    linkcolor=black,
    filecolor=magenta,      
    urlcolor=cyan,
}
\usepackage{geometry}
\def\XXint#1#2#3{{\setbox0=\hbox{$#1{#2#3}{\int}$ }
\vcenter{\hbox{$#2#3$ }}\kern-.6\wd0}}

\usepackage{amsthm}
\newtheorem{theorem}{Theorem}[section]

\newtheorem{lemma}[theorem]{Lemma}

\newtheorem{rem}{Remark}[section]

\numberwithin{equation}{section}

\renewcommand{\Im}{\operatorname{Im}}

\newcommand{\R}{\mathbb R}

\newcommand{\T}{\mathbb T}

\numberwithin{equation}{section}
\newcommand\norm[1]{\left\lVert#1\right\rVert}

\title[Well-posedness of periodic fractional NLS]{On the well-posedness of the periodic\\ fractional  Schr\"{o}dinger equation}

\author[B. Sanchez]{Beckett Sanchez}
\address{Department of Mathematics, Brown University, 151 Thayer Street, Providence, RI 02912}
\curraddr{}
\email{beckett\_sanchez@brown.edu}

\author[O. Riaño]{Oscar Riaño}
\address{Departamento de Matem\'aticas, Universidad Nacional de Colombia, Bogot\'a, Colombia. Carrera 30 Calle 45, Ciudad Universitaria, Bogot\'a, Colombia}
\curraddr{}
\email{ogrianoc@unal.edu.co}

\author[S. Roudenko]{Svetlana Roudenko}
\address{Department of Mathematics \& Statistics\\Florida International University,  Miami, FL, USA}
\curraddr{}
\email{sroudenko@fiu.edu}



\date{}

\begin{document}

\begin{abstract}
We consider the periodic fractional nonlinear Schr\"{o}dinger equation 
$$
iu_t -(-\Delta)^{\frac{s}{2}} u + \mathcal{N}(|u|)u=0, \quad x\in \mathbb{T}^N,\, \, t \in \mathbb R, \, \, s>0,
$$ 
where the nonlinearity term is expressed in two ways: the first one $\mathcal{N}\in C^J(\mathbb R^+)$, whose derivatives have a certain polynomial decay, e.g., $\mathcal{N}(|u|)=\log(|u|)$; the second one is given by a sum of powers, possibly infinite, 
$$
\mathcal{N}(|u|)  = \sum a_k |u|^{\gamma_k}, \quad \gamma_k \in \mathbb{R}, ~~ a_k \in \mathbb{C},
$$
which includes examples such as $\mathcal{N}(|u|) \, u =\frac{u}{|u|^{\gamma}},$ $\gamma>0$. By using standard properties of periodic Sobolev spaces $H^J(\mathbb{T}^N)$, $J>0$, we study the local well-posedness for the Cauchy problems of the above equations when initial data satisfy a non-vanishing condition $\inf\limits_{x\in \mathbb{T}^N}|u_0(x)|>0$. 
 \end{abstract}

\maketitle
\tableofcontents
\section{Introduction}

We consider the Cauchy problem associated to the periodic nonlinear fractional Schr\"odinger equation
\begin{equation}\label{NLS}
    \begin{cases}
        iu_t-(-\Delta)^{\frac{s}{2}} u + \mathcal{N}(|u|)u=0, \quad x\in\T^N, ~t\in\R,\\
        u(0,x)=u_0(x).
    \end{cases}
\end{equation}
where $u: \mathbb{R}\times \mathbb{T}^N\rightarrow \mathbb{C}$, dimension $N\geq 1$, the operator $(-\Delta)^{\frac{s}{2}}$ stands for the fractional Laplacian of order $s>0$, which is defined through the Fourier multiplier with symbol $|k|^{s}$, $k \in \mathbb{Z}^N$, and the potential or nonlinearity $\mathcal{N}(\cdot)$ is given by either of the following cases:
\smallskip

{\bf (I)} $\mathcal{N}(x)\in C^{J}(\mathbb{R}^{+})$ with $|\mathcal{N}^{(j)}(x)| \leq c_j |x|^{\gamma-j}$ for  some  $\gamma\in \mathbb{R}$, $c_j>0$, and each $j=1,\dots,J$.  

\qquad (Here, $\mathcal{N}^{(j)}(\cdot)=\frac{d^j \mathcal{N}}{dx^j}(\cdot)$ denotes the $j$th derivative of the function $\mathcal{N}$).

{\bf (II)} For $x \geq 0$, let $\mathcal{N}(x) = \sum \limits_{k=0}^{\infty}a_k x^{\gamma_k}, ~~a_k\in \mathbb{C}, \gamma_k\in \mathbb{R}$.

Setting $s=2$ in \eqref{NLS}, one recognizes the widely studied nonlinear Schr\"odinger (NLS) equation, which appears in various physical contexts such as propagation of laser beams, Bose-Einstein condensation, chemical phenomena, self-trapped beams in plasma, etc., see for instance \cite{Z1967, PhysRevE1997, SulemSulem1999, FP2000, MA1981, LamLippmanTappert1977} and numerous references therein.  The fundamental contributions in the periodic setting were started by Bourgain in \cite{B1993} (see review \cite{Gi1996}), and further developed in follow-up works such as \cite{S1997, Tao2006, EZ2008, IP2014, EGT2019} and many others. Fractional Schr\"odinger operators such as in \eqref{NLS} appear in  \cite{Laskin2000, Laskin2002} and rigorously derived as the continuum limit of some discrete physical systems with long-range lattice interactions in \cite{KLS2013}. They also appear in models for charge transport in biopolymers like the DNA, e.g. in \cite{DNA1999}, and other applications, and have been studied with various methods, e.g., see for example \cite{EZ2008,IP2014,HS2021}. Logarithmic nonlinearity in the NLS equation was introduced in \cite{bm1976} as a possible model in nonlinear wave mechanics, and was studied from the mathematical point of view (on a whole space), for example,  in \cite{CH1980, Caz1983, dAMS2014, Carles2018}. The second type of nonlinearity we consider resembles combined nonlinearities (at least when the powers are positive), such models have been studied, for example, in \cite{B1998, M2019, TVZ2007, CS2021, BFG2023} and references therein. In this paper, we look at the periodic setting of such potentials (and further their generalization such as allowing negative powers in the combined nonlinearities, or having infinitely many terms) and address the well-posedness theory from a point of view that does not rely on Strichartz estimates.

Concerning the well-posedness in $H^{J}(\mathbb{T}^N)$, in the case of the NLS equation \eqref{NLS} with $s=2$ and $\mathcal{N}(|u|)u= \lambda \,|u|^{\gamma}u$, $\gamma> 0$, $\lambda \in \{-1,1\}$,  we refer to 
\cite{B1993,Gi1996, KenigPonceVega1996, S1997, 
M2009, LinaresPonce2015} and references therein. 
The local and global well-posedness for the fractional Cauchy problem \eqref{NLS} with $\mathcal{N}(|u|)u=\lambda \, |u|^{\gamma}u$, $\gamma> 0$, $\lambda \in \{-1,1\}$ has been investigated, for instance, in \cite{CHKL2015, DET2016,GuoBolingYong2008,Thirouin2017, EGT2019}. 

The goal of this paper is to prove the existence and uniqueness (and continuous dependence) of solutions to the initial value problem \eqref{NLS} for {\it any} dispersion $s>0$, and a large class of nonlinearities, which includes some {\it analytic} functions (such as exponential, trigonometric, polynomial) and functions that are not smooth at the origin (e.g., logarithmic functions or negative powers). The main idea in studying 
such nonlinearities consists in obtaining solutions to \eqref{NLS} generated by non-vanishing initial conditions, i.e., $u_0\in H^J(\mathbb{T}^N)$, for which $\inf\limits_{x\in \mathbb{T}^N}|u_0(x)|>0$. Our motivation comes from the work of Cazenave and Naumkin \cite{CazNaum2016}, who studied the Cauchy problem for the NLS equation on the whole space
\begin{equation}\label{RnSchrodinger}
    \begin{cases}
        iu_t+\Delta u + |u|^{\gamma}u=0, \quad x\in\R^N, \, \, t\in\R, \, \gamma > 0, \\
        u(0,x)=u_0(x).
    \end{cases}
\end{equation}
For a certain class of non-vanishing initial conditions $u_0\in H^J(\mathbb{R}^N)$, in \cite{CazNaum2016}, it was shown the local existence of solutions of \eqref{RnSchrodinger} for an arbitrary $\gamma>0$, as well as the global existence and scattering when $\gamma>\frac{2}{N}$. The idea was based on considering a class of initial data such that
\begin{equation}\label{infcondRn}
\eta \inf_{x\in \mathbb{R}^N} |\langle x \rangle^ru_0(x)|\geq 1, 
\end{equation}
for some $\eta>0$, $r \in \mathbb N$, where $\langle x \rangle=(1+|x|^2)^{\frac{1}{2}}$, for which it was shown that there exists a time $T>0$ and a unique solution $u(t,x)$ of \eqref{RnSchrodinger} such that
\begin{equation*}
\eta \big(\inf_{t\in [-T,T], \, x\in \mathbb{R}^N  } |\langle x \rangle^ru(t,x)|\big)\geq c
\end{equation*}
for some $0<c<1$, 
which equivalently, can be written as 
\begin{equation*}
\inf_{t\in [-T,T], \, x\in \mathbb{R}^N  } |\langle x \rangle^ru(t,x)| > 0.
\end{equation*}

The above condition is convenient to deal with powers $\gamma>0$, where the nonlinearity $|u|^{\gamma}u$ in \eqref{RnSchrodinger} is not sufficiently regular at the origin (e.g., for $0<\gamma < 1$) to apply standard well-posedness techniques. In this regard, this method and the condition \eqref{infcondRn} have been used to study the well-posedness for different models with nonlinearities of low regularity, e.g., in the NLS-type equations \cite{AroraRianoRoudenko2022,RRR,CazHauNaum2020,CazNaum2018,LinaresPonceGleison2019,LinaresPonceGleison2019II}, and in the KdV-type equations \cite{FRRSY,LinaresMiyazakiPonce2019,Miyazaki2020}.

A natural analog of \eqref{infcondRn} for periodic functions  is given by
\begin{equation}\label{nonvanish}
\eta \inf_{x\in\T^N}|u_0(x)| \ge 1 
\end{equation}
for some $\eta>0$.

This paper sets out to investigate a local theory for solutions of \eqref{NLS} with different types of dispersion and nonlinearities
in the class determined by initial conditions $u_0\in H^J(\mathbb{T}^N)$ such that \eqref{nonvanish} holds. We note that in the periodic setting, the condition \eqref{nonvanish} is enough to show the existence of solutions in $H^J(\T^N)$ spaces without a polynomial  weight, as needed in the case of 
$\R^N$, e.g., see \cite{CazNaum2016, RRR, FRRSY}. As a consequence, when $s=2$, we consider a less restricted resolution space for the Cauchy problem \eqref{NLS}, unlike the weighted Sobolev space in \cite{CazNaum2016}. Moreover, restricting to the periodic case $\T^N$ allows for generalizations in nonlinearity, most notably $\frac{u}{|u|^{\gamma}}$, $\gamma>0$, as well as $\log (|u|)$. 

Regarding some invariants, the equation in \eqref{NLS} formally conserves the mass
\begin{equation*}
	M[u(t,x)] = \int_{\mathbb{T}^N}|u(t,x)|^2\, dx
\end{equation*}
and the energy
\begin{equation*}
	E[u(t,x)] = \frac{1}{2}\int_{\mathbb{T}^N}|(-\Delta)^{\frac{s}{4}} u(t,x)|^2\, dx -\int_{\mathbb{T}^N}\mathbf{N}(|u(t,x)|)\, dx,
\end{equation*}
where $\mathbf{N}(\tau)=\int_0^{\tau} \mathcal{N}(|\nu|)\nu\, d\nu$.
\smallskip

We now state our findings. Our first result establishes well-posedness for a general type of nonlinearities (I) determined by a function $\mathcal{N}\in C^{J}(\mathbb{R}^{+})$.

\begin{theorem}\label{mainTHM2}
Let $s>0$ and $J\in \mathbb{Z}^{+}$ be such that  $J>\frac{N}{2}+s$. Consider the Cauchy problem \eqref{NLS}, where $\mathcal{N}(\cdot)$ satisfies the following conditions:
\begin{itemize}
    \item[(i)] $\mathcal{N}(x)\in C^J(\mathbb{R}^{+})$,
    \item[(ii)] there exists $\gamma\in \mathbb{R}$ such that for all integer $n \in \{1,\dots, J \}$ 
    there exists a constant $c_n>0$, for which
    \begin{equation}\label{decayhypho}
       |\mathcal{N}^{(n)}(x)|\leq c_n|x|^{\gamma-n}, 
    \end{equation}
    for all $x>0$.
\end{itemize}
Let $u_0\in H^J(\T^N)$ be such that there exists an $\eta>0$, for which \eqref{nonvanish} holds. Then, there exist a time $T>0$ and a unique solution $u\in C([-T,T];H^J(\mathbb{T}^N))$ of \eqref{NLS} with initial condition $u_0$ such that
\begin{equation*}
    \inf_{t\in [-T,T], \, x\in \mathbb{T}^N}|u(t,x)|>0.
\end{equation*}
Moreover, the map $u_0 \mapsto u(t,\cdot)$ is continuous in the following sense: for any $0<\widetilde{T}<T$, there exists a neighborhood $V$ of $u_0$ in $H^J(\mathbb{T}^N)$, for which \eqref{nonvanish} holds,  and such that the map data-to-solution is Lipschitz continuous from $V$ into the class $C([-\widetilde{T},\widetilde{T}],H^{J}(\mathbb{T}^N))$.
\end{theorem}

The proof of the Theorem \ref{mainTHM2} is inspired by the methods developed in \cite{RRR} and \cite{CazNaum2016} for the NLS equation 
on $\mathbb R^N$. Our proof depends on a detailed analysis of the derivatives of the nonlinearity $\mathcal{N}(|u|)u$ and the difference $\mathcal{N}(|u|)u-\mathcal{N}(|v|)v$, where the conditions (i), (ii) in Theorem \ref{mainTHM2} as well as \eqref{nonvanish} are fundamental to achieve our results. We also remark that our results are obtained by using basic properties 
of the space $H^J(\mathbb{T}^N)$, where the regularity $J$ depends on the Sobolev embedding (see \eqref{SobEmb} below)  and the dispersion parameter $s>0$. For example, if $N=1$ and $0<s<\frac{1}{2}$, Theorem \ref{mainTHM2} shows the local well-posedness in $H^{1}(\mathbb{T})$. As a matter of fact, the condition $J>\frac{N}{2}+s$ assures that the quantity
\begin{equation}\label{linearapprox}
  \begin{aligned}
\sup_{t\in [-T,T], \, x\in \mathbb{T}^N}  \big| e^{-it(-\Delta)^{\frac{s}{2}}}u_0(x)-u_0(x) \big|  
  \end{aligned}  
\end{equation}
is small, provided that the time $T>0$ is sufficiently small (here, $e^{-it(-\Delta)^{\frac{s}{2}}}$ denotes the linear Schr\"odinger group on $\mathbb{T}^{N}$ associated to the linear equation 
$i\partial_t v-(-\Delta)^{\frac{s}{2}}v=0$, see further details in \S \ref{S:1}).    
Thus, if $u_0\in H^J(\mathbb{T}^N)$, $J>\frac{N}{2}+s$, satisfies \eqref{nonvanish}, the smallness of \eqref{linearapprox} establishes  a similar non-vanishing infimum condition for the linear evolution $e^{-it(-\Delta)^{\frac{s}{2}}}u_0$ for small times. This property is fundamental in the proof of Theorem \ref{mainTHM2}.

We next show several examples satisfying the assumptions on  
nonlinearity $\mathcal{N}(\cdot)$ in Theorem \ref{mainTHM2}.

$\bullet$ (\emph{Polynomial-type nonlinearity}). The function $\mathcal{N}(x)=x^{\gamma}$ for any real $\gamma$ (i.e., $\gamma\in \mathbb{R}$) and $x \geq 0$ satisfies the conditions (i) and (ii) of Theorem \ref{mainTHM2}. In particular, we prove the local well-posedness for the Cauchy problem for a potential with negative powers, that is, if $\nu=-\gamma$, then 
\begin{equation}\label{ptypenonl}
    \begin{cases}
        iu_t-(-\Delta)^{\frac{s}{2}} u - \frac{u}{|u|^{\nu}}=0, \quad x\in\T^N, ~t\in\R,\\
        u(0,x)=u_0(x)
    \end{cases}
\end{equation}
is locally well-posed for $\nu>0$ (as well as $\nu < 0$) and data satisfying \eqref{nonvanish}.

$\bullet$ (\emph{Logarithmic nonlinearities}). The functions  $\mathcal{N}(x)=\log(x)$ and  $\log(1+|x|^{\gamma})$, $\gamma\in \mathbb{R}$, satisfy the hypotheses of Theorem \ref{mainTHM2}. Consequently, Theorem \ref{mainTHM2} establishes the well-posedness for the initial value problems
\begin{equation}\label{logSchrodinger}
    \begin{cases}
        iu_t-(-\Delta)^{\frac{s}{2}} u + \log(|u|)u=0, \quad x\in\T^N, ~t\in\R, \qquad\\
        u(0,x)=u_0(x),
    \end{cases}
\end{equation}
and
\begin{equation}\label{logSchrodingerII}
    \begin{cases}
        iu_t-(-\Delta)^{\frac{s}{2}} u + \log(1+|u|^{\gamma})u=0, \quad x\in\T^N, ~t\in\R,\\
        u(0,x)=u_0(x).
    \end{cases}
\end{equation}

Our second main result is the local well-posedness for  the second (II) nonlinearity $\mathcal{N}(\cdot)$, which is given by a sum of some powers. Such a result allows us to consider another wide family of nonlinearities that are not covered by Theorem \ref{mainTHM2}. 

\begin{theorem}\label{mainTHM1}
Let $s>0$ and $J\in \mathbb{Z}^{+}$ be such that  $J>\frac{N}{2}+s$. Let the nonlinearity $\mathcal{N}(\cdot)$ in the equation \eqref{NLS} be as follows
$$
\mathcal{N}(|u|)=\sum\limits^\infty_{k=0}a_k|u|^{\gamma_k}, ~\gamma_k\in\mathbb{R}, ~a_k\in \mathbb{C}.
$$ 
Assume that the sequences $\{a_k\}$ and $\{\gamma_k\}$ are such that for any $R_0>0$ 
we have
\begin{equation}\label{convergenceeq}
\sum^{\infty}_{k=0}\sum_{p=0}^{J}C(\gamma_k,p)\,|a_k|
\left(  R_0^{\,|\gamma_k-2p|}+|\gamma_k-2p|R_0^{\,|\gamma_k-2p-1|}  \right)<\infty,
\end{equation}
where 
\begin{equation*}
C(\gamma_k,p) =\left\{
\begin{array}{cl}
1 &\qquad  ~ \emph{if} ~ p= 0, \\
\big| \footnotesize{\underbrace{\gamma_k(\gamma_k -2)...(\gamma_k-2(p-1) \, )}_{p-\text{times}} \big| 
} &\emph{ if } p=1, \dots, J.
\end{array}
\right.
\end{equation*}

Let $u_0\in H^J(\mathbb{T}^N)$ be such that there exist an $\eta>0$, for which \eqref{nonvanish} holds. Then, there exist a time $T>0$ and a unique solution $u\in C([-T,T];H^J(\mathbb{T}^N))$ of the Cauchy problem for \eqref{NLS} with the initial condition $u_0$ such that
\begin{equation*}
    \inf_{t\in [-T,T], \, x\in \mathbb{T}^N}|u(t,x)|>0.
\end{equation*}

Moreover, the map $u_0 \mapsto u(t,\cdot)$ is continuous in the following sense: for any $0<\widetilde{T}<T$, there exists a neighborhood $V$ of $u_0$ in $H^J(\mathbb{T}^N)$, for which \eqref{nonvanish} holds,  such that the map data-to-solution is Lipschitz continuous from $V$ into the class $C([-\widetilde{T},\widetilde{T}],H^{J}(\mathbb{T}^N))$.
\end{theorem}

The proof of Theorem \ref{mainTHM1} takes the idea from \cite{RRR}, which considered the one-dimensional nonlinear Schr\"odinger equation with combined nonlinearities posed on the real line. An important remark is that the choice of the regularity $J
\geq 0$ of the space $H^J(\mathbb{T}^N)$ in Theorem \ref{mainTHM1} is independent of the powers $\gamma_k$ of the terms in the sum of the nonlinearity, which is why we can use a series with different nonlinearities (and even infinitely many terms). Moreover, since the condition \eqref{nonvanish} does not involve weights in our case compared to the case of $\R^N$, see \eqref{infcondRn} (as in \cite{CazNaum2016}), the conclusion of Theorem \ref{mainTHM1} is valid for any real values of powers $\{\gamma_k\}$ (or $\gamma_k\in \mathbb{R}$ for any index $k$).

The following examples of nonlinear terms satisfy the assumptions of Theorem \ref{mainTHM1}, and hence, the NLS equation \eqref{NLS} with such nonlinearities is locally well-posed in $H^J(\mathbb T^N)$, provided the non-vanishing condition.

$\bullet $ (\emph{Combined nonlinearities}). Let $M\in \mathbb{Z}^{+}\cup\{0\}$, $a_k\in \mathbb{C}$, $\nu_k\in \mathbb{R}$, $0\leq k\leq M$, then
\begin{equation}\label{nonlineexamplecombi}
    \mathcal{N}(|u|)=\sum_{k=0}^{M}\frac{a_k}{|u|^{\nu_k}}
\end{equation}
satisfies the hypotheses of Theorem \ref{mainTHM1}. 

We note that if the nonlinearity $\mathcal{N}(\cdot)$ in \eqref{nonlineexamplecombi} has at least two distinct terms, i.e., 
there exist different integers $k$, $k'\geq 0$ such that $\nu_k\neq \nu_{k'}$ with $\nu_{k}\neq 0$, $\nu_{k'}\neq 0$, 
then such nonlinearity does not satisfy the assumptions of Theorem \ref{mainTHM2}, however, it does of Theorem \ref{mainTHM1}. This is one of the motivations to give two different conditions (I) and (II) on nonlinearity $\mathcal N (\cdot)$.

$\bullet$ (\emph{Exponential type nonlinearities}). Given $r,r_1,r_2\in \mathbb{R}$, Theorem \ref{mainTHM1} applies to the following examples: 
    \begin{equation}\label{E:e}
      \mathcal{N}(|u|)=e^{|u|^{r}}=\sum_{k=0}^{\infty}\frac{1}{k!}|u|^{rk},  
    \end{equation}    
    \begin{equation}\label{E:sin}
      \mathcal{N}(|u|)=\frac{\sin(|u|^{r_1})}{|u|^{r_2}}=\sum_{k=0}^{\infty}\frac{(-1)^k}{(2k+1)!}|u|^{2r_1k+r_1-r_2},  
    \end{equation}
and for a nonlinearity, where the sine function is replaced by the cosine function in \eqref{E:sin} with the corresponding series (as well as some other trigonometric functions). 
Well-posedness for such nonlinearities have been studied in \cite{RRR}. However, the results in \cite{RRR} only deal with the spatial variable on the one-dimensional real line 
and under the assumption that $r,r_1,r_2\geq 0$. 
Here, in the periodic setting, we consider arbitrary dimensions and the values $r, r_1, r_2$ can also be arbitrary real values.

\begin{rem} (i) This paper does not focus on the minimum or optimal regularity $J\in \mathbb{R}$ that guarantees the local well-posedness for solutions to \eqref{NLS} in the space $H^{J}(\mathbb{T}^N)$.  What we seek is to establish a well-posedness of solutions of \eqref{NLS} for a wide class of nonlinear terms (including $\log$ or inverse polynomials) satisfying the condition \eqref{nonvanish}. To our knowledge this is the first such result on $\mathbb{T}^N$ or in the periodic setting. It is possible that some of our results could be improved using different well-posedness techniques for \eqref{NLS} in the case $\mathcal{N}(|u|)u=|u|^{\gamma}u$, $\gamma>0$. It is not clear whether some of the standard techniques (such as via Strichartz estimates machinery) apply for nonlinearities of the form such as $\frac{u}{|u|^{\gamma}}$, $\gamma>0$. 

(ii) Observe that $\mathcal{N}(x)=\log(x)$ satisfies the conditions (i) and (ii) in Theorem \ref{mainTHM2}, but not the assumptions in Theorem \ref{mainTHM1}. Also note that $\mathcal{N}(x)=e^x$ (written in power series) satisfies the assumptions of Theorem \ref{mainTHM1}, but not those in Theorem \ref{mainTHM2}. 
\end{rem}

The paper is organized as follows: in Section \ref{S:1} we set the notation and preliminaries; in Section \ref{CJnonlinearity} we establish the proof of Theorem \ref{mainTHM2}, which concerns nonlinearities of type (I), i.e.,  $\mathcal{N}(x)\in C^{J}(\mathbb{R}^{+})$ and \eqref{decayhypho}. In this same section, we establish the key nonlinear estimates in $H^J(\mathbb{T}^N)$, see Lemma \ref{lemmanonlinCJ}.  In Section \ref{Seriesnonlinearity}, we prove Theorem \ref{mainTHM1}, which establishes the local well-posedness for nonlinearities of type (II), that is, the (possibly infinite) series.  
\smallskip

{\bf Acknowledgments.}
This project started during the Summer 2022 REU program ``AMRPU @ FIU" that took place at the Department of Mathematics and Statistics of Florida International University, and was supported by the NSA grant H982302210016 and NSF (REU Site) grant DMS-2050971 (PI: S. Roudenko). In particular, support of B.S. came from those grants. O.R. and S.R. were partially supported by the NSF grant DMS-2055130.


\section{Preliminaries and Notation}\label{S:1}

We use the standard multi-index notation, $\alpha=(\alpha_1,\dots,\alpha_N)\in \mathbb{N}_0^N$ (where $\mathbb{N}_0$ denotes the set of nonnegative integers), its length $|\alpha|=\sum_{j=1}^N \alpha_j$, its factorial $\alpha!=\alpha_1!\dots \alpha_N!$, and $\alpha\leq \beta$ means that $\alpha_j\leq \beta_j$ for all $j=1,\dots,N$. 
If $\partial_{x_j}$ denotes the partial derivative with respect to $x_j$, we write $\partial^{\alpha}=\partial_{x_1}^{\alpha_1}\dots\partial_{x_N}^{\alpha_n}$. Given $\alpha,\beta \in \mathbb{N}_0^{N}$, we define the binomial coefficient $\binom{\alpha}{\beta}=\frac{\alpha!}{\beta!(\alpha-\beta)!}$. For applications of the multi-index notation such as the Leibniz rule and Taylor expansion, see \cite[Chapter 1]{SRaymond1991}.

Given $1<p\leq \infty$, $L^{p}(\mathbb{T}^N)$ denotes the usual Lebesgue space equipped with the norm $\|f\|_{L^p}^p=\int_{\mathbb{T}^N}|f(x)|^p\, dx \equiv \int_{[0,2\pi]^N}|f(x)|^p\, dx$, with the usual modification when $p=\infty$. 

Given a function $f\in C^{\infty}(\mathbb{T}^N)$, the Fourier transform of $f$ is defined by 
\begin{equation*}
 \widehat{f}(k)=\frac{1}{(2\pi)^N}\int_{\mathbb{T}^N}  f(x) e^{-i k\cdot x}\, dx, \quad k \in \mathbb{Z}^N.  
\end{equation*}
For a rapidly decaying sequence $\{a_k\}_{k\in \mathbb{Z}^N}$, the inverse Fourier transform is defined as $\{a_k\}_{k\in \mathbb{Z}^N}^{\vee}=\sum_{k\in \mathbb{Z}^N} a_k e^{ik\cdot x}$.

Let $r\in \mathbb{R}$, the Sobolev space $H^r(\mathbb{T}^N)$ consists of all the periodic distributions (distributions on $C^{\infty}(\mathbb{T}^N)$) such that 
\begin{equation}\label{standarnormHs}
  \|f\|_{H^r}^2=\sum_{k\in \mathbb{Z}^N} \langle k \rangle^{2r}|\widehat{f}(k)|^2<\infty.  
\end{equation}

In this document, we consider Sobolev spaces with integer regularity $r=J\in \mathbb{N}_0$, for which we consider the norm, which is equivalent to \eqref{standarnormHs},
\begin{equation}\label{Integerder}
   \|f\|_{H^J}=\sum\limits_{0\leq |\beta|\leq J} \|\partial^{\beta}f\|_{L^2}. 
\end{equation}
The group of solutions $\{e^{-it(-\Delta)^{\frac{s}{2}}}\}$ is associated with the linear equation 
\begin{equation}\label{LinearEQ}
i\partial_t v-(-\Delta)^{\frac{s}{2}}v=0,
\end{equation}
which can be defined via the Fourier transform as
\begin{equation*}
\big( e^{-it(-\Delta)^{\frac{s}{2}}}f\big)(x)=\{e^{-i t|k|^s}\widehat{f}(k)\}^{\vee}(x)=\sum_{k\in \mathbb{Z}^N} e^{-i t|k|^s+ik\cdot x}\widehat{f}(k).   
\end{equation*}
Consequently, it follows that the integral or Duhamel formulation of \eqref{NLS} is
\begin{equation}\label{IntegralEquation}
  u(t)=e^{-it(-\Delta)^{\frac{s}{2}}}u_0+\int_0^t e^{-i(t-\tau)(-\Delta)^{\frac{s}{2}}}\mathcal{N}(|u(\tau)|)u(\tau)\, d\tau.
\end{equation}

The strategy to prove Theorems \ref{mainTHM2} and  \ref{mainTHM1} is to use a fixed point argument based on the integral equation \eqref{IntegralEquation} in some space $\mathcal Z$, defined later (e.g., see \eqref{zspace}), which would not only preserve the regularity of solutions but also would propagate a non-vanishing condition \eqref{nonvanish}.
To control different types of nonlinearities,  we, in particular,  use Sobolev embeddings, which we recall next.  

Let $0<r<\frac{N}{2}$, and $\frac{1}{q}=\frac{1}{2}-\frac{r}{N}$. If $f\in H^{r}(\mathbb{T}^N)$,  then $f\in L^q(\mathbb{T}^N)$, and there exists a constant $C_2 = C_2(r,N)>0$ such that
\begin{equation}\label{SobEmb}
    \|f\|_{L^q}\leq C_{2} \, \|f\|_{H^r}.
\end{equation}
For the proof of this result, refer to \cite[Corollary 1.2]{ArpadTadahiro2013}.

For completeness and to fix the constants, we also mention that when $q=\infty$ above, if $r>\frac{N}{2}$, we consider the Sobolev embedding $H^{r}(\mathbb{T}^N)\hookrightarrow L^{\infty}(\mathbb{T}^N)$, which states that there exists a constant $C_{1} = C_1(r,N)>0$ such that
\begin{equation}\label{SobembInft}
   \|f\|_{L^{\infty}}\leq C_{1}\|f\|_{H^r}, 
\end{equation}
for all $f\in H^r(\mathbb{T}^N)$. 
We will use the case of \eqref{SobembInft} with $r=J\in \mathbb{N}_0$ for $J>\frac{N}{2}$. 

Additionally, if $r>\frac{N}{2}$, one has that $H^r(\mathbb{T}^N)$ is a Banach algebra for the product of functions. That is, given $f,g\in H^r(\mathbb{T}^N)$, then $fg\in H^r(\mathbb{T}^N)$ with
\begin{equation*}
  \|fg\|_{H^r}\leq c_r\|f\|_{H^r}\|g\|_{H^r}.  
\end{equation*}


\section{Well-posedness results for nonlinearities of class $C^J$}\label{CJnonlinearity}

In this part, we prove Theorem \ref{mainTHM2}, the local well-posedness of the Cauchy problem \eqref{NLS} for the type (I) family of nonlinearities in the class $ C^{J}(\mathbb{R}^{+})$, which may have certain singularities at the origin. We divide the proof of Theorem \ref{mainTHM2} into two subsections: in the first one, we prove the estimates for the $H^J$-norm of the nonlinear part of \eqref{NLS}, and in the second one, we give the proof of Theorem \ref{mainTHM2}.

\subsection{Nonlinear estimates}

In the following lemma, we obtain nonlinear estimates for general functions $u,v\in H^J(\mathbb{T}^N)$ that satisfy the non-vanishing condition \eqref{nonvanish} for some $\eta>0$. 

We remark that in the proof of Theorem \ref{mainTHM2}, we connect Lemma \ref{lemmanonlinCJ} with the solutions $u$ of \eqref{NLS}, which will have initial data $u_0$, satisfying \eqref{nonvanish} (perhaps with a different size of $\eta$).

\begin{lemma}\label{lemmanonlinCJ}
Let $J\in \mathbb{Z}^{+}$ be such that $J>\frac{N}{2}$. Assume that the nonlinearity $\mathcal{N}(\cdot)$ satisfies the hypothesis (i) and (ii) in Theorem \ref{mainTHM2} for some fixed $\gamma\in \mathbb{R}$. Additionally, let $u,v\in H^J(\mathbb{T}^N)$ be such that for some $\eta>0$, the condition \eqref{nonvanish} holds for both $u$ and $v$. Then there exist constants $c = c(J, \gamma)$ and $K = K(J,\gamma)>0$ 
such that 
  \begin{equation}\label{nonlegeneraleq1}
\|\mathcal{N}(|u|)u\|_{H^J}\leq \mathcal G_1(\eta, \norm{u}_{H^J}) \|u\|_{H^J},  
  \end{equation}  
and
\begin{equation}\label{differenceCnon}
\begin{aligned}
\|&\mathcal{N}(|u|)u-\mathcal{N}(|v|)v\|_{H^J}
 \leq \mathcal G_2 (\eta,\|u\|_{H^J},\|v\|_{H^J})\|u-v\|_{H^J},    
\end{aligned} 
\end{equation}
where
\begin{equation}\label{g_1}
\mathcal G_1(\eta, \|u\|_{H^J}) = c \, \Big( 1+\eta+\|u\|_{H^J}+\mathcal{SN}(\eta,C_1\|u\|_{H^J})\Big)^{K(J,\gamma)},
\end{equation}
\begin{equation}\label{g_2}
\begin{aligned}
\mathcal G_2(\eta,\|u\|_{H^J},\|v\|_{H^J}) =& c \, \Big(1+\eta+\|u\|_{H^J}+\|v\|_{H^J}\\
&\quad +\mathcal{SN}(\eta,C_1\|u\|_{H^J})+\mathcal{SN}(\eta,C_1\|v\|_{H^J})\Big)^{K(J,\gamma)},      
\end{aligned}  
\end{equation}
and
\begin{equation}\label{supN}
\mathcal{SN}(\eta,C_1\|u\|_{H^J}) \stackrel{def}{=}
\sup_{x\in \big[\eta^{-1},\, C_1\|u\|_{H^J} \big] }|\mathcal{N}(x)|, 
\end{equation}
where $C_1 = C_1(J,N)>0$ is the constant in \eqref{SobembInft}.
\end{lemma}

\begin{proof}
We begin by noticing that the infimum condition \eqref{nonvanish} and the Sobolev embedding \eqref{SobembInft} yield
\begin{equation*}
   \eta^{-1}\leq |u|\leq \|u\|_{L^{\infty}}\leq C_1\|u\|_{H^J}.
\end{equation*}
Then, since $\mathcal{N}\in C(\mathbb{R}^{+})$, we deduce the following uniform bound
\begin{equation}\label{supest}
    |\mathcal{N}(|u|)|\leq \sup_{x\in [\eta^{-1},\, C_1\|u\|_{H^J}]}|\mathcal{N}(x)| \equiv \mathcal{SN}(\eta,C_1\|u\|_{H^J}) <\infty, 
\end{equation}
where the finiteness comes from $\mathcal{N}$ being continuous over the compact $[\eta^{-1},\, C_1\|u\|_{H^J}]\subset (0,\infty)$.

Moreover, the condition \eqref{nonvanish} and Sobolev embedding imply that for all $\widetilde{\gamma}\in \mathbb{R}$, we have
\begin{equation}\label{nonlinear1}
\begin{aligned}
|u|^{\widetilde{\gamma}}\leq \eta^{|\widetilde{\gamma}|}+\|u\|_{L^{\infty}}^{|\widetilde{\gamma}|}\leq c^{|\widetilde{\gamma}|}\big(\eta^{|\widetilde{\gamma}|}+\|u\|_{H^J}^{|\widetilde{\gamma}|}\big)
 \end{aligned}
\end{equation}
for some universal constant $c\geq 1$ independent of $u$, $\eta$, and $\widetilde{\gamma}$. More precisely, \eqref{nonlinear1} is obtained by putting together the estimate for $|u|^{\widetilde{\gamma}}$ in each of the cases $\widetilde{\gamma}\leq 0$ (using \eqref{nonvanish}), and $\widetilde{\gamma}>0$ (using Sobolev embedding \eqref{SobembInft}).

We now prove \eqref{nonlegeneraleq1}. Let $\alpha$ be a multi-index with order $|\alpha|\leq J$. By using the   Leibniz rule, we write
\begin{equation*}
  \partial^{\alpha}(\mathcal{N}(|u|)u)=\sum_{\beta\leq \alpha}\binom{\alpha}{\beta}\partial^{\beta}(\mathcal{N}(|u|))\,\partial^{\alpha-\beta}u.
\end{equation*}
When $|\beta|\geq 1$, an inductive argument on the size of the multi-index $\beta$ yields the following identity
\begin{equation*}
 \begin{aligned}
 \partial^{\beta}(\mathcal{N}(|u|))=\sum_{p=1}^{|
\beta|}\sum_{\substack{l_1+\dots+l_p=\beta\\ |l_j|\geq 1, \, j=1,\dots,p}}c_{l_1,\dots,l_p}\, \mathcal{N}^{(p)}(|u|)\,\partial^{l_1}(|u|)\dots \partial^{l_p}(|u|)
\end{aligned}   
\end{equation*}
for some constants $c_{l_1,\dots,l_p}$. 
Thus, we get
\begin{equation}\label{nonleestgeneraleq1}
 \begin{aligned}
     \partial^{\alpha}(\mathcal{N}(|u|)u)=&\mathcal{N}(|u|)\,\partial^{\alpha}u\\
     &+
     \sum_{0<\beta\leq \alpha}\sum_{p=1}^{|
\beta|}\sum_{\substack{l_1+\dots+l_p=\beta\\ |l_j|\geq 1, \, j=1,\dots,p}}c_{l_1,\dots,l_p}\, \mathcal{N}^{(p)}(|u|)\,\partial^{l_1}(|u|)\dots \partial^{l_p}(|u|)\,\partial^{\alpha-\beta}u.
 \end{aligned}   
\end{equation}
Note that we use `the zero convention' for the empty sum, i.e., if $\alpha=0$, then $\sum\limits_{0<\beta\leq \alpha}(\cdots)=0$, and hence, the expression in  \eqref{nonleestgeneraleq1} makes sense for all $0\leq |\alpha|\leq J$. Additionally, we also have that for a given multi-index $l$ with $|l|\geq 1$,
\begin{equation}\label{caseabsolform}
\partial^{l} (|u|)=\sum^{|l|}_{m=1}|u|^{1-2m}\sum_{\substack{l'_1+...+l'_m=l\\ |l'_j|\geq 1, \, j=1,\dots,m}} c_{l'_1,\dots,l'_m}\partial^{l'_1}(|u|^2) \dots \partial^{l'_m}(|u|^2),
\end{equation} 
for some constants $c_{l'_1,\dots,l'_m}$. We give a more general formula for the derivative of $|u|^{\gamma}$, $\gamma\in \mathbb{R}$, in \eqref{jderiv} below. Consequently, by \eqref{nonleestgeneraleq1}, \eqref{caseabsolform}, and using the Leibniz rule to compute $\partial^{l}(|u|^2)=\partial^{l}(u\overline{u})$, we have that $\partial^{\alpha}(\mathcal{N}(|u|)u)$ can be written as a sum of the factors 
\begin{equation}\label{Apart}
   \mathcal{A}_{\alpha}:= \mathcal{N}(|u|)\partial^{\alpha}u
\end{equation}
and
\begin{equation}\label{Bpart}
\mathcal{B}_{\beta,p}:= \mathcal{N}^{(p)}(|u|)\Big(\prod_{j=1}^p|u|^{1-2m_j}\partial^{l_{j,1,1}}u\overline{\partial^{l_{j,1,2}}u}\dots \partial^{l_{j,m_j,1}}u\overline{\partial^{l_{j,m_j,2}}u}\Big)\partial^{\alpha-\beta}u,  
\end{equation}
where 
\begin{equation}\label{restriccoeff}
\left\{\begin{aligned}
&0<\beta\leq \alpha,\qquad 1\leq p\leq |\beta|,\\
&l_1+\dots+l_p=\beta, \quad |l_j|\geq 1 \, \, \text{ for all } j=1,\dots p,\\  
&1\leq m_j\leq |l_j|, \,\,\text{ for each } j=1,\dots,p,\\
&l_{j,1,1}+l_{j,1,2}+\dots+l_{j,m_j,1}+l_{j,m_j,2}=l_j, \,  \, \text{ for each } j=1,\dots, p,\\
&|l_{j,r,1}|+|l_{j,r,2}|\geq 1,\,\, \text{for each } j=1,\dots, p \, \text{ and } 1\leq r\leq m_j.
\end{aligned}\right.    
\end{equation}
Hence, since $\alpha$ with $0\leq |\alpha|\leq J$ is arbitrary, to prove \eqref{nonlegeneraleq1}, it suffices to estimate the $L^2$-norms of $\mathcal{A}_{\alpha}$ and $\mathcal{B}_{\beta,p}$. By \eqref{supest} and  H\"older's inequality, we bound
\begin{equation}\label{Aestim}
\begin{aligned}
\|\mathcal{A}_{\alpha}\|_{L^2}\leq \|\mathcal{N}(|u|)\|_{L^{\infty}}\|\partial^{\alpha}u\|_{L^2}\leq \mathcal{SN}(\eta,C_1\|u\|_{H^J})\|u\|_{H^J}.
\end{aligned}    
\end{equation}
On the other hand, we use \eqref{decayhypho} and \eqref{nonlinear1} to infer
\begin{equation}\label{Bestim}
\begin{aligned}
\|\mathcal{B}_{\beta,p}\|_{L^2}
\leq & \, c \, \big\| |u|^{\gamma-p}\Big(\prod_{j=1}^p|u|^{1-2m_j}\partial^{l_{j,1,1}}u\overline{\partial^{l_{j,1,2}}u}\dots \partial^{l_{j,m_j,1}}u\overline{\partial^{l_{j,m_j,2}}u}\Big)\partial^{\alpha-\beta} u \big\|_{L^2}  \\
\leq & \, c \, \Bigg(\eta^{\big|\gamma-2\sum\limits_{j=1}^p m_j \big|}+\|u\|_{H^J}^{\big|\gamma-2\sum\limits_{j=1}^{p}m_j \big|} \,  \Bigg)\\
&\quad\times \underbrace{\big\| \Big(\prod_{j=1}^p \partial^{l_{j,1,1}}u\overline{\partial^{l_{j,1,2}}u}\dots \partial^{l_{j,m_j,1}}u\overline{\partial^{l_{j,m_j,2}}u}\Big)\partial^{\alpha-\beta}u \big\|_{L^2}.}_{\|\widetilde{\mathcal{B}}_{\beta,p}\|_{L^2}}
\end{aligned}    
\end{equation}

To estimate $\|\widetilde{\mathcal{B}}_{\beta,p}\|_{L^2}$, we consider three cases, which are given by comparison between the number ${J-\frac{N}{2}>0}$ and 
the highest order derivative among the multi-indices $l_{j,1,1}, \, l_{j,1,2}, \, \dots$,  $l_{j,m_j,1}, \, l_{j,m_j,2}, \, \alpha-\beta$.
\smallskip

\underline{Case 1: The highest order derivative $<J-\frac{N}{2}$}. In this case, the order of each derivative of the functions in $\widetilde{\mathcal{B}}_{\beta,p}$ is less than or equal to $J-\lfloor \frac{N}{2}\rfloor-1$. Thus, taking the $L^{\infty}$-norm of all the factors in $\widetilde{\mathcal{B}}_{\beta,p}$ and using the Sobolev embedding $H^{\lfloor\frac{N}{2} \rfloor+1}(\mathbb{T}^N)\hookrightarrow L^{\infty}(\mathbb{T}^N)$ (recalling also that $\mathbb{T}^N$ has a finite measure, $\|1\|_{L^2} = (2\pi)^{\frac{N}2}$), we obtain 
\begin{equation}\label{Bestim2.0}
\begin{aligned}
\|\widetilde{\mathcal{B}}_{\beta,p}  \|_{L^2}
\leq & (2\pi)^{\frac{N}2}
\Big(\prod_{j=1}^p\|\partial^{l_{j,1,1}}u\|_{L^{\infty}}\|\partial^{l_{j,1,2}}u\|_{L^{\infty}}\dots \|\partial^{l_{j,m_j,1}}u\|_{L^{\infty}}\|\partial^{l_{j,m_j,2}}u\|_{L^{\infty}}\Big)\|\partial^{\alpha-\beta}u\|_{L^{\infty}}\\
\leq & \, c \, \Big(\|u\|_{H^J}\Big)^{2\sum_{j=1}^p m_j}\|u\|_{H^J}.
\end{aligned}  
\end{equation}

\underline{Case 2: The highest order derivative $>J-\frac{N}{2}$}. If the derivative of the highest order in $\|\widetilde{\mathcal{B}}_{\beta,p}\|_{L^2}$ is equal to $J$, then it must follow that $p=1$, $|\beta|=J$, $\alpha=\beta$, and $m_1=1$. Thus, in this case, the Sobolev embedding yields
\begin{equation}\label{intercase}
\begin{aligned}
\|\widetilde{\mathcal{B}}_{\beta,p}\|_{L^2}\leq c\|\big(|u|^{-1}u \partial^{\beta} u\big) u\|_{L^2}\leq c\|u\|_{L^{\infty}}\|\partial^{\beta}u\|_{L^2}\leq c\|u\|_{H^J}^2.
\end{aligned}  
\end{equation}
Next, we assume that the derivative of the highest order is strictly less than $J$ and strictly greater than $J-\frac{N}{2}$. Consequently, to simplify the notation and separate the derivative of the highest order, we use H\"older's inequality to write 
\begin{equation}\label{decompHolder}
\begin{aligned}
 \|\widetilde{\mathcal{B}}_{\beta,p}\|_{L^2}\leq c  \bigg(\prod_{j=1}^p \|\partial^{\widetilde{l}_{j,1}}u\|_{L^q}\dots\|\partial^{\widetilde{l}_{j,2m_j}}u\|_{L^q} \bigg)\|\partial^{\widetilde{\beta}}u\|_{L^{q^{\ast}}}
\end{aligned}    
\end{equation}
for some multi-indices $J-\frac{N}{2}<|\widetilde{\beta}|<J$, $\Big|\sum\limits_{j=1}^p\sum\limits_{k=1}^{2m_j}\widetilde{l}_{j,k}\Big|\leq J-|\widetilde{\beta}|<\frac{N}{2}$, and 
\begin{equation*}
\left\{ \begin{aligned}
 & \frac{1}{q^{\ast}}=\frac{1}{2}-\frac{(J-|\widetilde{\beta}|)}{N},\\
 & \frac{1}{q}=\frac{1}{2}-\frac{r}{N},
 \end{aligned}\right.   
\end{equation*}
with
\begin{equation}
    r:=\bigg(\frac{N}{2}\bigg(\sum_{j=1}^p 2m_j\bigg)-(J-|\widetilde{\beta}|)\bigg)\bigg(\sum_{j=1}^p 2m_j\bigg)^{-1}.
\end{equation}
One can check that $0<r< \frac{N}{2}$ and $\Big(\sum\limits_{j=1}^p 2m_j\Big)\frac{1}{q}+\frac{1}{q^{\ast}}=\frac{1}{2}$. Then, using \eqref{SobEmb} in \eqref{decompHolder} yields
\begin{equation}\label{Bestim2.1}
\begin{aligned}
 \|\widetilde{\mathcal{B}}_{\beta,p}\|_{L^2}\leq c \Big(\|u\|_{H^J}\Big)^{2\sum_{j=1}^p m_j} \|u\|_{H^J}.
\end{aligned}    
\end{equation}
\underline{Case 3: The highest order derivative $=J-\frac{N}{2}$}. This case can be treated as we did when $J>\frac{N}{2}$. Notice that $J-\frac{N}{2}$ being an integer implies that the dimension $N$ must be an even integer. Now, taking the derivative of the highest order, an application of H\"older's inequality yields
\begin{equation}\label{eqdecaneq1}
\begin{aligned}
 \|\widetilde{\mathcal{B}}_{\beta,p}\|_{L^2}\leq c  \bigg(\prod_{j=1}^p \|\partial^{\widetilde{l}_{j,1}}u\|_{L^q}\dots\|\partial^{\widetilde{l}_{j,2m_j}}u\|_{L^q} \bigg)\|\partial^{\widetilde{\beta}}u\|_{L^{q}}
\end{aligned}    
\end{equation}
for some multi-indices with $|\widetilde{\beta}|=J-\frac{N}{2}$,  $\Big|\sum\limits_{j=1}^p\sum\limits_{k=1}^{2m_j}\widetilde{l}_{j,k}\Big|\leq J-|\widetilde{\beta}|\leq\frac{N}{2}$, and $q=2\big(1+\sum\limits_{j=1}^p 2m_j\big)$. Thus, setting $0<r=\frac{N\big(\sum_{j=1}^p2m_j\big)}{2\big(1+\sum_{j=1}^p2m_j\big)}<\frac{N}{2}$, using that all the derivatives $|\widetilde{l}_{j,k}|\leq J-\frac{N}{2}$, and applying the Sobolev embedding \eqref{SobEmb} to the right-hand side of \eqref{eqdecaneq1}, we get
\begin{equation}\label{Bestim2}
\begin{aligned}
 \|\widetilde{\mathcal{B}}_{\beta,p}\|_{L^2}\leq c \Big(\|u\|_{H^J}\Big)^{2\sum_{j=1}^p m_j} \|u\|_{H^J},
\end{aligned}    
\end{equation}
where we have also used that $J-\frac{N}{2}+r\leq J$. This completes 
Case 3. 
Gathering together the estimates from Cases 1, 2 and 3 above, we obtain the estimate \eqref{Bestim}.

Consequently, setting $p_2=\sum_{j=1}^p m_j\leq J$, and 
\begin{equation}\label{Kdef}
    K(J,\gamma):=\max_{0\leq p_1,p_2\leq J}\bigg\{\Big|\gamma-2p_2\Big|+2p_2,|\gamma-p_1-1|+|p_1-2p_2|+2p_2+1,|p_1-2p_2-1|+2p_2+2\bigg\},
\end{equation} 
it follows that \eqref{nonlegeneraleq1} is a consequence of the estimates \eqref{Aestim}, \eqref{Bestim}, \eqref{Bestim2.0}, \eqref{Bestim2.1}, \eqref{Bestim2}, and the fact that $\alpha$ is arbitrary. Here, we included more terms in the above constant $K(J,\gamma)$ as it will also be needed for the proof of 
\eqref{g_2}.

Next, we show \eqref{differenceCnon}. From the representations \eqref{Apart} and \eqref{Bpart}, and re-grouping terms, we write the difference $\mathcal{N}(|u|)u-\mathcal{N}(|v|)v$ as a sum of the following factors:
\begin{equation*}
 \begin{aligned}
& \mathcal{C}_1:=(\mathcal{N}(|u|)-\mathcal{N}(|v|))\partial^{\alpha}u,\\
 & \mathcal{C}_2:=\mathcal{N}(|v|)(\partial^{\alpha}u-\partial^{\alpha}v),\\
 &\mathcal{C}_3:= \big(\mathcal{N}^{(p)}(|u|)-\mathcal{N}^{(p)}(|v|)\big)|u|^{p-\sum_{j=1}^{p}2m_j}\Big(\prod_{j=1}^p\partial^{l_{j,1,1}}u\overline{\partial^{l_{j,1,2}}u}\dots \partial^{l_{j,m_j,1}}u\overline{\partial^{l_{j,m_j,2}}u}\Big)\partial^{\alpha-\beta}u,\\
&\mathcal{C}_{4}:=\mathcal{N}(|v|)\Big(|u|^{p-2\sum\limits_{j=1}^p m_j}-|v|^{p-2\sum\limits_{j=1}^{p}m_j}\Big)\prod_{j=1}^p \Big(\partial^{l_{j,1,1}}u\overline{\partial^{l_{j,1,2}}u}\dots \partial^{l_{j,m_j,1}}u\overline{\partial^{l_{j,m_j,2}}u}\Big)\partial^{\alpha-\beta}u,\\
&\mathcal{C}_5:=\mathcal{N}(|v|)|v|^{p-2\sum\limits_{j=1}^{p}m_j}\prod_{j=1}^p \Big(\partial^{l_{j,1,1}}v\overline{\partial^{l_{j,1,2}}v}\dots \partial^{l_{j,m_j,1}}v\overline{\partial^{l_{j,m_j,2}}v}\Big)\partial^{\alpha-\beta}(u-v),
 \end{aligned}   
\end{equation*}
where we also assumed the same restrictions as in \eqref{restriccoeff}, and furthermore, for $1\leq k\leq p$ and $1\leq \lambda\leq m_k$, the difference $\mathcal{N}(|u|)u-\mathcal{N}(|v|)v$ also incorporates the terms
\begin{equation*}
\begin{aligned}
 &\mathcal{C}_{6,k,\lambda}:=\mathcal{N}(|v|)|v|^{p-2\sum\limits_{j=1}^{p}m_j}\Big(\prod_{j=1}^{k-1}\partial^{l_{j,1,1}}v\overline{\partial^{l_{j,1,2}}v}\dots \partial^{l_{j,m_j,1}}v\overline{\partial^{l_{j,m_j,2}}v}\Big)\\
&\hspace{2cm}\times\Big(\partial^{l_{k,1,1}}v\overline{\partial^{l_{k,1,2}}v}\dots \partial^{l_k,\lambda,1}(u-v)\overline{\partial^{l_k,\lambda,2}u}\dots \partial^{l_{k,m_k,1}}u\overline{\partial^{l_{k,m_k,2}}u}\\
&\hspace{3cm}+\partial^{l_{k,1,1}}v\overline{\partial^{l_{k,1,2}}v}\dots \partial^{l_k,\lambda,1}v\overline{\partial^{l_k,\lambda,2}(u-v)}\dots \partial^{l_{k,m_k,1}}u\overline{\partial^{l_{k,m_k,2}}u}\Big)\\
&\hspace{2cm}\times\Big(\prod_{j=k+1}^{p}\partial^{l_{j,1,1}}u\overline{\partial^{l_{j,1,2}}u}\dots \partial^{l_{j,m_k,1}}u\overline{\partial^{l_{j,m_k,2}}u}\Big)\partial^{\alpha-\beta}u,   
\end{aligned}
\end{equation*}
where we use the convention $\prod_{j=1}^0(\cdots)=\prod_{j={p+1}}^p(\cdots)=1$. Before estimating the previous factors, we first notice that given $\widetilde{\gamma}\in \mathbb{R}$, the mean value inequality, \eqref{nonvanish} and the Sobolev embedding  yield
\begin{equation}\label{differengammaest}
\begin{aligned}
  ||u|^{\widetilde{\gamma}}-|v|^{\widetilde{\gamma}}|\leq & |\gamma|(\eta^{|\widetilde{\gamma}-1|}+|u|^{|\widetilde{\gamma}-1|}+|v|^{|\widetilde{\gamma}-1|})|u-v|  \\ 
  \leq & C_1^{|\widetilde{\gamma}-1|}|\widetilde{\gamma}|(\eta^{|\widetilde{\gamma}-1|}+\|u\|_{H^J}^{|\widetilde{\gamma}-1|}+\|v\|_{H^J}^{|\widetilde{\gamma}-1|})\|u-v\|_{H^J},
\end{aligned}
\end{equation}
where $C_1$ is as in \eqref{SobembInft}. Additionally, letting $0\leq p'\leq J$ to be an integer, a further application of the mean value inequality and the decay condition \eqref{decayhypho} allows us to write 
\begin{equation}\label{Ndiffere}
 \begin{aligned}
|\mathcal{N}^{(p')}(|u|)-\mathcal{N}^{(p')}(|u|)| \leq & c\big(|u|^{\gamma-(p'+1)}+|v|^{\gamma-(p'+1)}\big)|u-v|\\
\leq & c\big(\eta^{|\gamma-p'-1|}+\|u\|_{L^{\infty}}^{|\gamma-p'-1|}+\|v\|_{L^{\infty}}^{|\gamma-p'-1|}\big)|u-v|\\
\leq & c\big(\eta^{|\gamma-p'-1|}+\|u\|_{H^J}^{|\gamma-p'-1|}+\|v\|_{H^J}^{|\gamma-p'-1|}\big)\|u-v\|_{H^J},
 \end{aligned}   
\end{equation}
where we used the Sobolev embedding \eqref{SobembInft} and the fact that $u$ and $v$ satisfy \eqref{nonvanish}. Now, we deal with the factors $\mathcal{C}_j$, $\mathcal{C}_{6,k,\lambda}$. By \eqref{Ndiffere}, and \eqref{supest}, we get
\begin{equation}\label{estimC1C2}
  \begin{aligned}
\|\mathcal{C}_1\|_{L^2}+\|\mathcal{C}_2\|_{L^2}\leq & \|\mathcal{N}(|u|)-\mathcal{N}(|v|)\|_{L^{\infty}}\|\partial^{\beta}u\|_{L^2}+\|\mathcal{N}(|v|)\|_{L^{\infty}}\|\partial^{\alpha}u-\partial^{\alpha}v\|_{L^2}\\
  \leq & c\big(\eta^{|\gamma-2|}+\|u\|_{H^J}^{|\gamma-2|}+\|v\|_{H^J}^{|\gamma-2|}\big)\|u\|_{H^J}\|u-v\|_{H^J}\\
  &+\big(\mathcal{SN}(\eta,C_1\|u\|_{H^J})+\mathcal{SN}(\eta,C_1\|v\|_{H^J})\big)\|u-v\|_{H^J}.    
  \end{aligned}  
\end{equation}
We use \eqref{Ndiffere} and the ideas from \eqref{Bestim}-\eqref{Bestim2} to infer
\begin{equation}\label{estimC3}
\begin{aligned}
 \|\mathcal{C}_3\|_{L^2}
  \leq \,  & \|\mathcal{N}^{(p)}(|u|)-\mathcal{N}^{(p)}(|v|)\|_{L^{\infty}}\\
  & \qquad \times\||u|^{p-2\sum_{j=1}^p m_j}\Big(\prod_{j=1}^p\partial^{l_{j,1,1}}u\overline{\partial^{l_{j,1,2}}u}\dots \partial^{l_{j,m_j,1}}u\overline{\partial^{l_{j,m_j,2}}u}\Big)\partial^{\alpha-\beta}u\|_{L^2}\\
  \leq & c\,\big(\eta^{|\gamma-p-1|}+\|u\|_{H^J}^{|\gamma-p-1|}+\|v\|_{H^J}^{|\gamma-p-1|}\big)(\|u\|_{H^J}^{\big|p-2\sum_{j=1}^p m_j\big|+2\sum_{j=1}^p m_j+1})\|u-v\|_{H^J}.   
\end{aligned}
\end{equation}
Similarly as above, \eqref{supest}, \eqref{differengammaest} and the estimates for $\widetilde{\mathcal{B}}_{\beta,p}$ from \eqref{Bestim}-\eqref{Bestim2} yield
\begin{equation}\label{estimC4}
 \begin{aligned}
\|\mathcal{C}_4\|\leq \, & \|\mathcal{N}(|v|)\|_{L^{\infty}}
\Big\| |u|^{p-2\sum\limits_{j=1}^p m_j}-|v|^{p-2\sum\limits_{j=1}^{p}m_j} \Big\|_{L^{\infty}}\\
&\hspace{2cm}\times \Big\| \prod_{j=1}^p \Big(\partial^{l_{j,1,1}}u\overline{\partial^{l_{j,1,2}}u}\dots \partial^{l_{j,m_j,1}}u\overline{\partial^{l_{j,m_j,2}}u}\Big)\partial^{\alpha-\beta}u \Big\|_{L^2},\\
\leq &\, c\big(\mathcal{SN}(\eta,C_1\|u\|_{H^J})+\mathcal{SN}(\eta,C_1\|v\|_{H^J})\big)\\
&\qquad \times \Big(\eta^{\big|p-2\sum_{j=1}^p m_j-1\big|}+\|u\|_{H^J}^{\big|p-2\sum_{j=1}^p m_j-1\big|}+\|v\|_{H^J}^{\big|p-2\sum_{j=1}^p m_j-1\big|}\Big)\\
&\qquad \, \,\times(\|u\|_{H^J}^{1+2\sum_{j=1}^p m_j})\|u-v\|_{H^J}.
 \end{aligned}   
\end{equation}
Next, we use \eqref{supest}, and \eqref{nonlinear1} to bound
\begin{equation*}
 \begin{aligned}
\|\mathcal{C}_5\|_{L^5}\leq \, & \|\mathcal{N}(|v|)\|_{L^{\infty}} \big\||v|^{p-2\sum\limits_{j=1}^{p}m_j} \big\|_{L^{\infty}}\\
&\hspace{2cm}\times \Big\|\prod_{j=1}^p \Big(\partial^{l_{j,1,1}}v\overline{\partial^{l_{j,1,2}}v}\dots \partial^{l_{j,m_j,1}}v\overline{\partial^{l_{j,m_j,2}}v}\Big)\partial^{\alpha-\beta}(u-v) \Big\|_{L^2}\\
\leq & c\big(\mathcal{SN}(\eta,C_1\|u\|_{H^J})+\mathcal{SN}(\eta,C_1\|v\|_{H^J})\big)\\
&\qquad \times\Big(\eta^{\big|p-2\sum_{j=1}^p m_j\big|}+\|u\|_{H^J}^{\big|p-2\sum_{j=1}^p m_j\big|}+\|v\|_{H^J}^{\big|p-2\sum_{j=1}^p m_j\big|}\Big)\\
&\hspace{2cm}\times \Big\|\prod_{j=1}^p \Big(\partial^{l_{j,1,1}}v\overline{\partial^{l_{j,1,2}}v}\dots \partial^{l_{j,m_j,1}}v\overline{\partial^{l_{j,m_j,2}}v}\Big)\partial^{\alpha-\beta}(u-v) \Big\|_{L^2}.
 \end{aligned}   
\end{equation*}
Using the fact that $J>\frac{N}{2}$, by similar arguments to those in \eqref{Bestim2.0}-\eqref{Bestim2}, where we split into cases comparing the size of the highest order derivative with $J-\frac{N}{2}$, we obtain 
\begin{equation*}
 \begin{aligned}
\Big \|\prod_{j=1}^p \Big(\partial^{l_{j,1,1}}v\overline{\partial^{l_{j,1,2}}v}\dots \partial^{l_{j,m_j,1}}v\overline{\partial^{l_{j,m_j,2}}v}\Big)\partial^{\alpha-\beta}(u-v) \Big\|_{L^2}\leq \Big(\|v\|_{H^J}^{2\sum_{j=1}^p m_j}\Big)\|u-v\|_{H^J}.     
 \end{aligned}   
\end{equation*}
Summarizing, we get
\begin{equation}\label{estimC5}
\begin{aligned}
\|\mathcal{C}_5 \|_{L^2}
  \leq & \, c\big(\mathcal{SN}(\eta,C_1\|u\|_{H^J})+\mathcal{SN}(\eta,C_1\|v\|_{H^J})\big)\\
&\, \, \times \Big(\eta^{\big|p-2\sum_{j=1}^p m_j\big|}+\|u\|_{H^J}^{\big|p-2\sum_{j=1}^p m_j\big|}+\|v\|_{H^J}^{\big|p-2\sum_{j=1}^p m_j\big|}\Big)\Big(\|v\|_{H^J}^{2\sum_{j=1}^p m_j}\Big)\|u-v\|_{H^J}.
\end{aligned}    
\end{equation}
A similar argument to that in the estimate of $\mathcal{C}_5$, using \eqref{supest}, \eqref{nonlinear1}, and the ideas in \eqref{Bestim2.0}-\eqref{Bestim2} yield
\begin{equation}\label{estimC6}
\begin{aligned}
\|\mathcal{C}_{6,k,\lambda}\|_{L^2} \leq & c\big(\mathcal{SN}(\eta,C_1\|u\|_{H^J})+\mathcal{SN}(\eta,C_1\|v\|_{H^J})\big)\\
&\, \, \times\big(\eta^{\big|p-2\sum_{j=1}^p m_j\big|}+\|u\|_{H^J}^{\big|p-2\sum_{j=1}^p m_j\big|}+\|v\|_{H^J}^{\big|p-2\sum_{j=1}^p m_j\big|}\big)\\
&\,\,\times\big(\|v\|_{H^J}^{2\sum\limits_{j=1}^{k-1}m_j}+\|u\|_{H^J}^{2\sum\limits_{j=k+1}^{p}m_j}\big)\|u\|_{H^J}\|u-v\|_{H^J}.  
\end{aligned}    
\end{equation}
Gathering \eqref{estimC1C2}, \eqref{estimC3}, \eqref{estimC4}, \eqref{estimC5}, \eqref{estimC6} and the definition of $K(J,\gamma)>0$ in \eqref{Kdef}, we complete the proof of \eqref{differenceCnon}. 
\end{proof}


\subsection{Proof of Theorem \ref{mainTHM2}}

We are now in the position to prove Theorem \ref{mainTHM2}. We show first the existence and uniqueness of solutions for the Cauchy problem \eqref{NLS} with the type (I) nonlinearity, of the form \eqref{decayhypho}. We then show the continuous dependence as a consequence of our arguments in the proof of the existence below.
\smallskip

\underline{Existence of solutions}. We use the contraction mapping principle based on the integral formulation \eqref{IntegralEquation}. 
More precisely, given $s>0$, let $J\in \mathbb{Z}^{+}$ be such that $J>\frac{N}{2}+s$, and let $u_0\in H^J(\mathbb{T}^N)$, for which \eqref{nonvanish} holds for some $\eta>0$. For $R>0$ and $T>0$ to be chosen later, we consider the complete space $\mathcal{Z}_{T,R}$ defined as

\begin{equation}\label{zspace}
\mathcal{Z}_{T,R}=
\begin{cases}
    & u\in C([-T,T], H^J(\T^N)): \\  
    & ~~\sup \limits_{t\in[-T,T]} \norm{u(t)}_{H^J} \le R, \\
    & ~~\eta \big(\inf \limits_{t\in [-T,T],\, x\in \mathbb{T}^N  } |u(t,x)|\big)\geq \frac12
\end{cases}
\end{equation}
equipped with the distance function $\sup\limits_{t\in [-T,T]}\|u(t)-v(t)\|_{H^J}$, for $u,v\in \mathcal{Z}_{T,R}$. 

Note that the infimum condition in the definition of the space $\mathcal{Z}_{T,R}$ is bounded below by $\frac12$
as we cannot guarantee that at a later time the solution of \eqref{NLS} with the initial condition $u_0$ remains non-zero with the same constant as in \eqref{nonvanish}.

Now, given $u\in \mathcal{Z}_{T,R}$, for $t\in [-T,T]$ we consider  
\begin{equation*}
 \begin{aligned}
  \Phi(u)(t)=e^{-it(-\Delta)^{\frac{s}{2}}}u_0+\int_0^t e^{-i(t-\tau)(-\Delta)^{\frac{s}{2}}}\mathcal{N}(|u(\tau)|)u(\tau)\, d\tau,   
 \end{aligned}   
\end{equation*}
where $\mathcal N(|u|)$ takes the form \eqref{decayhypho}. We show that there exist $R>0$ and a time $T>0$ such that $\Phi$ defines a contraction on the complete metric space $\mathcal{Z}_{T,R}$. We divide our argument into two main parts: first, we show that for $R=2\|u_0\|_{H^J}$, there exists a time $T>0$, small enough, such that $\Phi$ maps $\mathcal{Z}_{T,R}$ into itself. Then, we show that $\Phi$ defines a contraction on $\mathcal{Z}_{T,R}$ provided that $T>0$ is again small enough.
\smallskip

\noindent$\bullet$ \underline{$\Phi: \mathcal{Z}_{T,R}\rightarrow \mathcal{Z}_{T,R}$}. 
Since $\{e^{-it (-\Delta)^{\frac{s}{2}}}\}$ defines a group of isometries on $H^J(\mathbb{T}^N)$, we apply Lemma \ref{lemmanonlinCJ} to get
\begin{equation}\label{integralest1}
 \begin{aligned}
   \|\Phi(u)(t)\|_{H^J}\leq & \|u_0\|_{H^J}+\Big|\int_{0}^{t}\|\mathcal{N}(|u(\tau)|)u(\tau)\|_{H^J}\, d\tau\Big|\\
   \leq& \|u_0\|_{H^J}+c \, T\Big(\sup_{t\in [-T,T]}\mathcal{G}_1(2\eta,\|u(t)\|_{H^J})\Big)\Big(\sup\limits_{t\in[-T,T]}\|u(t)\|_{H^J}\Big)\\
   \leq & \|u_0\|_{H^J}+c\,T\mathcal{G}_1(2\eta,R)R,
 \end{aligned}   
\end{equation}
where $\mathcal G_1(\cdot,\cdot)$ is defined in \eqref{g_1} with $c_1>0$, $K_{J,\gamma}>0$ being the fixed constants provided by \eqref{nonlegeneraleq1}, \eqref{differenceCnon},  and \eqref{supN}. Setting $R=2\|u_0\|_{H^J}$ and $T>0$ sufficiently small such that
\begin{align}\label{boundcondition}
c \, T\mathcal G_1(2\eta,R)\, R\le \frac{R}{2},
\end{align}
we find that $\sup\limits_{t\in[-T,T]}\|\Phi(u)(t)\|_{H^J} \le R$. Moreover, \eqref{integralest1} and the fact that $u\in C([-T,T];H^J(\mathbb T^N))$ implies $\Phi(u)(t)\in C([-T,T];H^J(\mathbb T^N))$. 

Next, we check the infimum condition. We write
\begin{equation}\label{integraldiff}
    \begin{aligned}
\Phi(u)(t) &= u_0 + \Big(e^{-it(-\Delta)^{\frac{s}{2}}}u_0-u_0\Big)+\int_0^t e^{-i(t-\tau)(-\Delta)^{\frac{s}{2}}}\mathcal{N}(|u(\tau)|)u(\tau)\, d\tau\\
&= u_0 -i\int_{0}^te^{-i \tau (-\Delta)^{\frac{s}{2}} }(-\Delta)^{\frac{s}{2}} u_0\, d\tau+\int_0^t e^{-i(t-\tau)(-\Delta)^{\frac{s}{2}}}\mathcal{N}(|u(\tau)|)u(\tau)\, d\tau.
 \end{aligned}   
\end{equation}
Since $J> \frac{N}{2}+s$, $s>0$, by Sobolev embedding 
\begin{equation}\label{linearnonvanscond}
 \Big|\int_0^t e^{-i\tau (-\Delta)^{\frac{s}{2}} }(-\Delta)^{\frac{s}{2}} u_0\, d\tau\Big|\leq T\sup \limits _{t\in [-T,T]}\|e^{-it(-\Delta)^{\frac{s}{2}}}(-\Delta)^{\frac{s}{2}} u_0\|_{L^{\infty}}\leq c \, T\|u_0\|_{H^J}.  
\end{equation}
Applying \eqref{nonvanish}, \eqref{linearnonvanscond}, and \eqref{integralest1}, we obtain the estimate 
\begin{equation*}
 \begin{aligned}
 |\Phi(u)(t)|\geq & \, \eta^{-1}- c\,T\|u_0\|_{H^J}-c\,T\mathcal{G}_1(2\eta,R)\, R.
 \end{aligned}
\end{equation*}
Consequently, taking $T>0$ sufficiently small such that 
\begin{equation}\label{infcondition}
c \, T\frac{R}{2}+c\, T\mathcal{G}_1(2\eta,R)R \leq \frac{1}{2}\eta^{-1},
\end{equation}
yields 
\begin{equation}\label{infintegraleqCJ}
2\eta\left(\inf\limits_{t\in[-T,T],x\in \mathbb T^N}|\Phi(u)(t)| \right) \ge 1.
\end{equation}
We conclude that for $R=2\|u_0\|_{H^J}$, there exists a $T>0$ such that \eqref{boundcondition} and \eqref{infcondition} hold, proving that $\Phi(u)$ defines a map from $\mathcal Z_{T,R}$ to itself. 
\smallskip

\noindent$\bullet$ \underline{$\Phi$ is a contraction}. Let $u,v\in\mathcal Z_{T,R}$. Observe that by \eqref{differenceCnon}, for any $t \in [-T,T]$
\begin{equation*}
 \begin{aligned}
\|\Phi(u)(t)-\Phi(v)(t)\|_{H^J} 
   \leq & \Big|\int_{0}^{t}\|\mathcal{N}(|u(\tau)|)u(\tau)-\mathcal{N}(|v(\tau)|)v(\tau)\|_{H^J}\, d\tau \Big|\\
   \leq& \,c\,T\Big(\sup_{t\in [-T,T]}\mathcal{G}_2(2\eta,\|u(t)\|_{H^J},\|v(t)\|_{H^J})\Big)\Big(\sup_{t\in [-T,T]}\|u(t)-v(t)\|_{H^J}\Big)\\
   \leq& \,c\,T\,\mathcal{G}_2(2\eta,R,R)\Big(\sup_{t\in [-T,T]}\|u(t)-v(t)\|_{H^J}\Big)
 \end{aligned}   
\end{equation*}
for $\mathcal G_2(\cdot,\cdot,\cdot)$ defined in \eqref{g_2}.  
Thus, taking $T>0$ small enough such that
\begin{equation*}
cT\mathcal{G}_2(2\eta,R,R)\leq \frac{1}{2},   
\end{equation*}
yields
\begin{equation}\label{diffintegraleqCJ} 
 \begin{aligned}
\|\Phi(u)(t)-\Phi(v)(t)\|_{H^J}\leq \frac{1}{2}\Big(\sup_{t\in [-T,T]}\|u(t)-v(t)\|_{H^J}\Big),
 \end{aligned}   
\end{equation}
concluding that $\Phi$ is a contraction. 

Gathering the previous estimates, there exists a time $T>0$ sufficiently small such that \eqref{integralest1}, \eqref{infintegraleqCJ}, and  \eqref{diffintegraleqCJ} imply that $\Phi: \mathcal{Z}_{T,R}\rightarrow \mathcal{Z}_{T,R}$ defines a contraction on $\mathcal{Z}_{T,R}$. Consequently, Banach fixed-point theorem yields the existence of a solution $u\in C([-T,T]; H^J(\mathbb T^N))$ of the Cauchy problem \eqref{NLS} with the initial condition $u_0$. This completes the proof of existence.
\smallskip

\underline{Uniqueness of solutions}. The uniqueness follows by energy estimates at the $L^2$-level. Let $u_1, u_2 \in C([-T,T];H^J(\mathbb{T}^N))$ be two solutions of \eqref{NLS} with the same initial condition $u_0$, for which the condition \eqref{nonvanish} holds, i.e., there exists $\eta>0$ such that 
\begin{equation}\label{uniquenessinfcond}
   \eta \big(\inf \limits_{t\in[-T,T], \, x\in \mathbb{T}^N}|u_j(t,x)|\big)\geq 1, \qquad j=1,2.
\end{equation}
Let $M>0$ be such that
\begin{equation}\label{Kbound}
  \sup\limits_{t\in[-T,T]}\|u(t)\|_{H^J}+\sup\limits_{t\in[-T,T]}\|v(t)\|_{H^J}\leq M.  
\end{equation}
Setting $w=u-v$, we obtain that $w$ solves the initial value problem
\begin{equation}\label{uniqueeq}
\begin{cases}
i w_t - (-\Delta)^{\frac{s}{2}} w +\big(\mathcal{N}(|u|)u-\mathcal{N}(|v|)v\big)=0, \\
w(0,x)=0.
\end{cases}
\end{equation}
In what follows, we establish uniqueness for times $t > 0$, by using \eqref{uniqueeq} and energy estimates. For $t < 0$, defining $\widetilde{w}(t)=w(-t)$, we apply the same argument for positive times below to the equation for $\widetilde{w}$, proving uniqueness for negative times.

Multiplying the equation in \eqref{uniqueeq} by $\overline{w}$, integrating over $\mathbb{T}^N$, and taking the imaginary part to the resulting expression, we deduce
\begin{equation}\label{unique1}
\begin{aligned}
 \frac{1}{2}\frac{d}{dt}\int_{\mathbb{T}^N} |w|^2\, dx=&-\Im \int_{\mathbb{T}^N}\big(\mathcal{N}(|u|)u-\mathcal{N}(|v|)v\big)\overline{w}\, dx\\
 \leq & \int_{\mathbb{T}^N}\big|\mathcal{N}(|u|)-\mathcal{N}(|v|)\big|u||w|\, dx+\int_{\mathbb{T}^N}|\mathcal{N}(|v|)||w|^2\, dx,
\end{aligned}
\end{equation}
where the middle term of \eqref{uniqueeq} vanishes due to the periodicity of $w$ and the fact that it is real-valued. Next, we apply \eqref{Ndiffere}, and H\"older's inequality to find
\begin{equation}\label{unique2}
\begin{aligned}
 \int_{\mathbb{T}^N}&\big|\mathcal{N}(|u|)-\mathcal{N}(|v|)\big|u||w|\, dx\\
 &\leq  c\, \big(1+\eta+\sup_{t\in[-T,T]}\|u(t)\|_{H^J}+\sup_{t\in[-T,T]}\|v\|_{H^J}\big)^{|\gamma-1|}\big(\sup_{t\in[-T,T]}\|u(t)\|_{H^J}\big)\int_{\mathbb{T}^N}|w|^2\, dx  \\
 &\leq c\,\mathcal{G}_3(\eta,M)\int_{\mathbb{T}^N}|w|^2\, dx,
\end{aligned}    
\end{equation}
and using \eqref{supest}, we find
\begin{equation}\label{unique3}
\begin{aligned}
 \int_{\mathbb{T}^N}&|\mathcal{N}(|v|)||w|^2\, dx\\
 &\leq  c\, \big(\sup_{t\in[-T,T]}\mathcal{SN}(\eta,C_1\|u(t)\|_{H^J})+\sup_{t\in[-T,T]}\mathcal{SN}(\eta,C_1\|v(t)\|_{H^J})\big)\int_{\mathbb{T}^N}|w|^2\, dx  \\
 &\leq c\, \mathcal{G}_3(\eta,M)\int_{\mathbb{T}^N}|w|^2\, dx,
\end{aligned}    
\end{equation}
where $C_1>0$ is the fixed constant provided by \eqref{SobembInft}, and
\begin{equation*}
   \mathcal{G}_3(\eta,M) \stackrel{def}{=} (1+
   \eta+2M)^{|\gamma-1|}M+2\mathcal{SN}(\eta,C_1 M). 
\end{equation*}
Hence, plugging \eqref{unique2} and \eqref{unique3} into \eqref{unique1}, and applying Gronwall's inequality, we obtain that $w=0$, that is, $u=v$. The proof of uniqueness is complete.
\smallskip

\underline{Continuous dependence}. Let $u_0\in H^{J}(\mathbb{T}^N)$ satisfy \eqref{nonvanish} for some $\eta>0$. By the existence and uniqueness parts in the above proof, there exist a time $T>0$ and a unique solution $u\in C([-T,T];H^{J}(\mathbb{T}^N))$ of \eqref{NLS} with the initial condition $u_0$ such that
\begin{equation*}
   \inf_{t\in[-T,T],\, x\in \mathbb{T}^N}|u(t,x)|>0. 
\end{equation*}
Recall that from the existence part above, when $R=2\|u_0\|_{H^J}$, we can choose $T>0$ such that
\begin{equation}\label{condcontrac}
\left\{ \begin{aligned}
   &c\,T\mathcal{G}_1(2\eta,R)R\leq \frac{R}{2}, \\
   &c\,T\mathcal{G}_2(2\eta,R,R)\leq \frac{1}{2},\\
   &c\,T\frac{R}{2}+cT\mathcal{G}_1(2\eta,R)R\leq \frac{1}{2}\eta^{-1}.
 \end{aligned}  \right. 
\end{equation}
The above conditions establish that $\Phi$ is a contraction on $\mathcal{Z}_{T,R}$ into itself, and from that we also obtain that
${\sup\limits_{t\in [-T,T]}\|u(t)\|_{H^J}\leq R}$.

Fixing $\widetilde{T} \in (0,T)$ fixed, we define $V=\{v_0\in H^J(\mathbb{T}^N): \|v_0-u_0\|_{H^J}<\epsilon\}$, where $\epsilon>0$ is such that
\begin{equation}\label{tildacond}
\left\{\begin{aligned}
&\epsilon \leq \frac{1}{4 c}\eta^{-1}\\
  &\epsilon+c\,\widetilde{T}\mathcal{G}_1(2\eta,R)R\leq \frac{R}{2},\\
  &\epsilon +c\,\widetilde{T}\frac{R}{2}-c\,\widetilde{T}\mathcal{G}_1(2\eta,R)R\leq \frac{1}{2}\eta^{-1},
\end{aligned}\right. 
\end{equation}
where $\mathcal{G}_1$ is defined in \eqref{g_1}, and $c>0$ is the same constant in \eqref{condcontrac} (which also depends on the Sobolev embedding $H^{J}(\mathbb{T}^N)\hookrightarrow L^{\infty}(\mathbb{T}^N)$ as in \eqref{SobembInft}). Observe that by the Sobolev embedding and our choice of $\epsilon$, we get
\begin{equation*}
   \begin{aligned}
       |v_0|\geq |u_0|-|u_0-v|\geq \eta^{-1}-c\,\|u_0-v_0\|_{H^J}\geq \frac{3}{4}\eta^{-1}.
   \end{aligned} 
\end{equation*}
Thus, for each $v_0\in V$,
\begin{equation*}
    \frac{4}{3}\,\eta\Big(\inf\limits_{x\in \mathbb{T}^N}|v_0(x)|\Big)\geq 1.
\end{equation*}
Consequently, using \eqref{tildacond}, we can apply the same argument in the existence part above to show that for any $v_0\in V$, there exists a unique fixed point $v\in \mathcal{Z}_{\widetilde{T},R}$ (see \eqref{zspace}) of the function
\begin{equation}
  \widetilde{\Phi}_{v_0}(v)(t)=e^{-it(-\Delta)^{\frac{s}{2}}}v_0+\int_0^t e^{-i(t-\tau)(-\Delta)^{\frac{s}{2}}}\mathcal{N}(|v(\tau)|)v(\tau)\,d\tau.  
\end{equation}
In other words, by uniqueness, given $v_0\in V$, there exists a unique solution $v\in C([-\widetilde{T},\widetilde{T}];H^J(\mathbb{T}^N))$ to the Cauchy problem \eqref{NLS} with the initial condition $v_0$. This shows the existence of the map data-to-solution from $V$ into the class $C([-\widetilde{T},\widetilde{T}];H^J(\mathbb{T}^N))$. 

To prove Lipschitz continuity, let $u_0,v_0\in V$, and $u,v\in C([-\widetilde{T},\widetilde{T}];H^J(\mathbb{T}^N))$ be the corresponding solutions to \eqref{NLS} with $u(0)=u_0$, $v(0)=v_0$. By \eqref{diffintegraleqCJ}, we get
\begin{equation*}
\begin{aligned}
 \sup\limits_{t\in [-\widetilde{T},\widetilde{T}]}\|u(t)-v(t)\|_{H^J}\leq \|u_0-v_0\|_{H^J}+c\,\widetilde{T}\mathcal{G}_2(2\eta,R,R)\Big(\sup\limits_{t\in [-\widetilde{T},\widetilde{T}]}\|u(t)-v(t)\|_{H^J}\Big),   
\end{aligned}    
\end{equation*}
where $\mathcal{G}_2$ is defined in \eqref{g_2}. Since $\widetilde{T}<T$, we have $c\,\widetilde{T}\mathcal{G}_2(2\eta,R,R)\leq c \,T\mathcal{G}_2(2\eta,R,R)\leq \frac{1}{2}$. Hence,  the previous inequality shows
\begin{equation*}
\begin{aligned}
 \sup\limits_{t\in [-\widetilde{T},\widetilde{T}]}\|u(t)-v(t)\|_{H^J}\leq 2\,\|u_0-v_0\|_{H^J},   
\end{aligned}    
\end{equation*}
which completes the justification of the continuous dependence property and finished the proof of Theorem \ref{mainTHM2}.

\section{Well-posedness results for series nonlinearities}\label{Seriesnonlinearity}

In this section, we give the proof of Theorem \ref{mainTHM1}, which shows the local well-posedness for the Cauchy problem \eqref{NLS} with the type (II) nonlinearity, given as a series. We first obtain preliminary estimates for the polynomial-type nonlinearities such as $|u|^\gamma u$, then prove Theorem \ref{mainTHM1}.

\subsection{Nonlinear estimates}

In the following lemma, we study the $H^{J}$-norm of $|u|^{\gamma}u$ and the difference $|u|^{\gamma}u-|v|^{\gamma}v$ (for $u,v \in H^{J}(\mathbb{T}^N)$ such that \eqref{nonvanish} holds for some $\eta>0$). A key observation here is that we require precise estimates of the constants in terms of $\gamma$. This turns out to be fundamental for studying the initial value problem \eqref{NLS} with infinitely many terms in the combined nonlinearity $\mathcal{N}(x)=\sum \limits_{k=0}^{\infty}a_k x^{\gamma_k}$, for arbitrary real-valued $\gamma_k$'s, similar to the case of the real line in \cite{RRR}.

\begin{lemma}\label{nonlinearestimates}
Let $J\in \mathbb{Z}^{+}$ be such that $J>\frac{N}{2}$. Let $u,v\in H^J(\T^N)$ be such that there exists $\eta>0$, for which \eqref{nonvanish} holds for both $u$ and $v$. Then for any given $\gamma\in\mathbb{R}$, we have 
\begin{equation}\label{nonestimates}
\begin{aligned}
 \norm{|u|^{\gamma}u}_{H^J} \le & \, \Gamma_1(\gamma,\eta,\|u\|_{H^J}), 
\end{aligned}
\end{equation}
and
\begin{equation}\label{diffnonlinear}
\norm{u|u|^{\gamma}-|v|^{\gamma}v}_{H^J}\leq \Gamma_2(\gamma,\eta,\|u\|_{H^J},\|v\|_{H^J})\|u-v\|_{H^J},
\end{equation}
where 
\begin{equation*}
\begin{aligned}
\Gamma_1(\gamma,\eta,\|u\|_{H^J}):=& c_0c_1^{|\gamma|}( \eta^{|\gamma|} + \norm{u}^{|\gamma|}_{H^J} )\norm{u}_{H^J}  \\
&+\sum_{0<|\beta|\leq J}\sum_{p=1}^{|\beta|}\sum_{l_1+\dots+l_p=\beta}c_{\beta} c_1^{|\gamma-2p|}|c_{l_1,\dots,l_p}| (\eta^{|\gamma -2p|}+\norm{u}_{H^J}^{|\gamma -2p|}) \norm{u}^{2p+1}_{H^J},
\end{aligned}    
\end{equation*}
and
\begin{equation*}
\begin{aligned}
\Gamma_2 (\gamma,\eta&,\|u\|_{H^J},\|v\|_{H^J})\\
:=& c_0c_1^{|\gamma-1|}|\gamma|(\eta^{|\gamma-1|}+\|u\|_{H^J}^{|\gamma-1|}+\|v\|_{H^J}^{|\gamma-1|})(\|u\|_{H^J}+\|v\|_{H^J})\\
&+c_0c_1^{|\gamma|}(\eta^{|\gamma|}+\|u\|_{H^J}^{|\gamma|}+\|v\|_{H^J}^{|\gamma|})\\
&+\sum_{0<|\beta|\leq J}\sum_{p=1}^{|\beta|}\sum_{l_1+\dots+l_p=\beta}c_{\beta}|c_{l_1,\dots,l_p}|\\
&\quad \times \Big( c_1^{|\gamma-2p-1|}|\gamma-2p|(\eta^{|\gamma-2p-1|}+\norm{u}^{|\gamma-2p-1|}_{H^J}+\norm{v}^{|\gamma-2p-1|}_{H^J} )(\|u\|_{H^J}^{2p+1}+\|v\|_{H^J}^{2p+1})\\
& \qquad + c_1^{|\gamma-2p|}(\eta^{|\gamma-2p|}+\norm{u}^{|\gamma-2p|}_{H^J}+\norm{v}^{|\gamma-2p|}_{H^J} )(\|u\|_{H^J}^{2p}+\|v\|_{H^J}^{2p})\Big),
\end{aligned}
\end{equation*}
with
\begin{align}\label{coefficientscond}
 |c_{l_1,\dots,l_p}| \le 2^{|\beta|} \underbrace{\big|(\gamma/2)(\gamma/2 -1)\dots(\gamma/2-(p-1)) \big|}_{p-\text{times}},
\end{align} 
 $c_1\geq 1$ is any fixed  constant such that $\|f\|_{L^{\infty}}\leq c_1\|f\|_{H^{\lfloor\frac{N}{2}\rfloor+1}}$, for all $f\in H^{\lfloor\frac{N}{2}\rfloor+1}(\mathbb{T}^N)$, $c_0>0$ depends on $J$, and $c_{\beta}>0$ depends on $\beta$, for each $0\leq |\beta|\leq J$.
\end{lemma}

\begin{rem}
The proof of Lemma \ref{nonlinearestimates} follows the same strategy developed for the NLS equation on a real line in \cite{RRR}, however, here we deal with the periodic setting in any spatial dimension. Moreover, we observe that setting $\mathcal{N}(x)=x^{\gamma}$, one can deduce Lemma \ref{nonlinearestimates} as a consequence of Lemma \ref{lemmanonlinCJ}. However, in order to make the dependence of the constants in terms of $\gamma$ more transparent (for the purpose of the infinite sum), we provide the proof of Lemma \ref{nonlinearestimates} below.
\end{rem}

\begin{proof}[Proof of Lemma \ref{nonestimates}]

Let $0\leq |\alpha|\leq J$. By Leibniz rule, we have 
\begin{equation}\label{Leibniz}
    \partial^{\alpha}\left(|u|^{\gamma} u\right)=\sum_{\beta\leq \alpha}\binom{\alpha}{\beta}\partial^{\beta}(|u|^{\gamma})\partial^{\alpha-\beta}u.
\end{equation} 
When $|\beta|\geq 1$, the general derivative of $|u|^{\gamma}$ can be expressed as
\begin{equation}\label{jderiv}
\partial^{\beta} (|u|^{\gamma})=\sum^{|\beta|}_{p=1}|u|^{\gamma-2p}\sum_{\substack{l_1+...+l_p=\beta\\ l_j\geq 1, \, j=1,\dots,p}} c_{l_1,\dots,l_p}\partial^{l_1}(|u|^2)...\partial^{l_p}(|u|^2),
\end{equation} 
where the coefficients $c_{l_1,\dots,l_p}$ satisfy \eqref{coefficientscond}. We remark that an inductive argument on the size of the multi-index $|\beta|\geq 1$ yields \eqref{jderiv}. Thus, we can express \eqref{Leibniz} as
\begin{equation}\label{derivinterpre}
\begin{aligned}
\partial^{\alpha}(|u|^{\gamma}u)=&\underbrace{|u|^{\gamma} \partial^{\alpha} u}_{\mathcal{A}}\\
&+ \underbrace{\sum_{0<\beta\leq \alpha}\sum_{p=1}^{|\beta|}\sum_{\substack{l_1+\dots+l_p=\beta\\ l_j\geq 1, \, j=1,\dots,p}}\binom{\alpha}{\beta}c_{l_1,...,l_p} |u|^{\gamma-2p} \partial^{l_1}(|u|^2)...\partial^{l_p}(|u|^2)\partial^{\alpha-\beta}u}_{\mathcal{B}},
\end{aligned}
\end{equation}
where $c_{l_1,...,l_p}$ satisfies \eqref{coefficientscond}, and we assume that the empty sum is equal to zero, e.g., when $\alpha=0$, we have $\sum_{0<\beta\leq \alpha}(\cdots)=0$. Consequently, the deduction of  \eqref{nonestimates} is reduced to obtaining the control of the $L^2$-norm in terms of $\mathcal{A}$ and $\mathcal{B}$ above. By H\"older's inequality and \eqref{nonlinear1}, we have
\begin{equation}\label{decompB}
\begin{aligned}
\|\mathcal{A}\|_{L^2} \leq \norm{|u|^{\gamma}}_{L^\infty}\norm{\partial^{\alpha} u}_{L^2}\leq c^{|\gamma|}(\eta^{|\gamma|}+\|u\|_{H^{J}}^{|\gamma|})\|u\|_{H^J}.
\end{aligned}    
\end{equation}
Next, we observe that by Leibniz rule, $\mathcal{B}$ can be decomposed as the sum of terms of the form
\begin{equation*}
   |u|^{\gamma-2p}\partial^{l_{1,1}}u\overline{\partial^{l_{1,2}}u}\dots \partial^{l_{p,1}}u\overline{\partial^{l_{p,2}}u}\partial^{l_{p+1}}u, 
\end{equation*}
where given $0<\beta\leq \alpha$,
\begin{equation*}
\left\{\begin{aligned}
&1\leq p\leq |\beta|,\\
&l_{1,1}+l_{1,2}+\dots+l_{p,1}+l_{p,2}+l_{p+1}=\alpha.
\end{aligned}\right.
\end{equation*}
An application of \eqref{nonlinear1} allows us to deduce
\begin{equation}\label{prodnonl}
\begin{aligned}
\|   |u|^{\gamma-2p}&\partial^{l_{1,1}}u\overline{\partial^{l_{1,2}}u}\dots \partial^{l_{p,1}}u\overline{\partial^{l_{p,2}}u}\partial^{l_{p+1}}u\|_{L^2}\\
&\leq c^{|\gamma-2p|}(\eta^{|\gamma-2p|}+\|u\|^{|\gamma-2p|}_{H^J})\|\partial^{l_{1,1}}u\overline{\partial^{l_{1,2}}u}\dots \partial^{l_{p,1}}u\overline{\partial^{l_{p,2}}u}\partial^{l_{p+1}}u\|_{L^2}\\
&=: c^{|\gamma-2p|}(\eta^{|\gamma-2p|}+\|u\|^{|\gamma-2p|}_{H^J})\|\mathcal{B}_1\|_{L^2},
\end{aligned}    
\end{equation}
where the constant $c>0$ above depends only on the constant of Sobolev embedding $H^{\lfloor\frac{N}{2} \rfloor+1}(\mathbb{T}^N)\hookrightarrow L^{\infty}(\mathbb{T}^N)$. The estimate of $\mathcal{B}_1$ is similar to that of $\widetilde{\mathcal{B}}_{\beta,p}$ in \eqref{Bestim}. If the derivative of the highest order among $l_{1,1},l_{1,2},\dots,l_{p,1},l_{p,2},l_{p+1}$ has order less than $J-\frac{N}{2}$, the same ideas in \eqref{Bestim2.0} show
\begin{equation}\label{firstsecondcases}
 \begin{aligned}
   \|\mathcal{B}_1\|_{L^2}\leq c^{2p+1}\|u\|_{H^{J}}^{2p+1},  
 \end{aligned}   
\end{equation}
where the constant above depends on the embedding $H^{\lfloor\frac{N}{2} \rfloor+1}(\mathbb{T}^N)\hookrightarrow L^{\infty}(\mathbb{T}^N)$. Now, if the derivative of the highest order is equal to $J$, we have that $p=1$ in $\mathcal{B}_1$, and the estimate is similar to that in \eqref{intercase}. Now, we assume that the derivative of the highest order is  strictly greater than $J-\frac{N}{2}$ and strictly less than $J$. Thus, without loss of generality, we assume $J-\frac{N}{2}<|l_{p+1}|<J$, then $\sum_{k=1}^p|l_{k,1}|+|l_{k,2}|\leq J-|l_{p+1}|<\frac{N}{2}$. We apply  H\"older's inequality to obtain
\begin{equation}\label{upperboundB1}
\begin{aligned}
\|\mathcal{B}_1\|_{L^2}\leq \|\partial^{l_{1,1}}u\|_{L^{q}}\|\partial^{l_{1,2}}u\|_{L^q}\dots \|\partial^{l_{p,1}}u\|_{L^q}\|\partial^{l_{p,2}}u\|_{L^q}\|\partial^{l_{p+1}}u\|_{L^{q_1}}    
\end{aligned}    
\end{equation}
where
\begin{equation*}
 \left\{\begin{aligned}
 &\frac{1}{q_1}=\frac{1}{2}-\frac{(J-|l_{p+1}|)}{N},\\
 &\frac{1}{q}=\frac{1}{2}-\frac{r}{N},
 \end{aligned}\right.   
\end{equation*}
and
\begin{equation*}
   r=\frac{pN-(J-|l_{p+1}|)}{2p}. 
\end{equation*}
Using that $0<r<\frac{N}{2}$ and $\frac{2p}{q}+\frac{1}{q_1}=\frac{1}{2}$, we apply the Sobolev embedding \eqref{SobEmb} in the inequality \eqref{upperboundB1} to get 
\begin{equation}\label{thirdcase}
    \|\mathcal{B}_1\|_{L^2}\leq (\widetilde{c})^{2p+1}\|u\|_{H^J}^{2p+1},
\end{equation}
here, the constant $\widetilde{c}>0$ depends on the Sobolev embeddings $H^{r}(\mathbb{T}^N)\hookrightarrow L^{q}(\mathbb{T}^N)$ and $H^{J-|l_{p+1}|}(\mathbb{T}^N)\hookrightarrow L^{q_1}(\mathbb{T}^N)$. Finally, when the derivative of the highest order is equal to $J-\frac{N}{2}$, the arguments in the deduction of \eqref{Bestim2} yield the same estimate in \eqref{thirdcase} with the constant depending on the embedding $H^{r}(\mathbb{T}^N)\hookrightarrow L^{q}(\mathbb{T}^N)$ with $q=2(2p+1)$ and $r=\frac{2pN}{2(2p+1)}$.

Collecting the previous estimates, it follows 
\begin{equation*}
 \begin{aligned}
  \|\mathcal{B}\|_{L^2}\leq \sum_{0<\beta\leq \alpha}\sum_{p=1}^{|\beta|}\sum_{l_1+\dots+l_p=\beta}\binom{\alpha}{\beta} c_1^{|\gamma-2p|}(\widetilde{c})^{2p+1} |c_{l_1,...,l_p}|(\eta^{|\gamma-2p|}+\|u\|^{|\gamma-2p|}_{H^J})\|u\|_{H^J}^{2p+1},  
 \end{aligned}   
\end{equation*}
where $c_{l_1,...,l_p}$ satisfies \eqref{coefficientscond}, $c_1\geq 1$ is any fixed  constant such that $\|f\|_{L^{\infty}}\leq c_1\|f\|_{H^{\lfloor\frac{N}{2}\rfloor+1}}$, for all $f\in H^{\lfloor\frac{N}{2}\rfloor+1}(\mathbb{T}^N)$, and $\widetilde{c}>0$ is such that \eqref{firstsecondcases} and \eqref{thirdcase} hold. The previous bounds for $\mathcal{A}$ and $\mathcal{B}$ establish the desired estimate for $\partial^{\alpha}(|u|^{\gamma}u)$, and since $0\leq |\alpha| \leq J$ is arbitrary, we get \eqref{nonestimates}. 
\smallskip

Next, we deal with the difference $|u|^{\gamma}u-|v|^{\gamma}v$, where $u,v\in H^J(\mathbb{T}^N)$ and both satisfy \eqref{nonvanish}. Let $\alpha$ be a multi-index with $0\leq |\alpha| \leq J$. Using \eqref{derivinterpre} and regrouping terms, we rewrite  $\partial^{\alpha}(|u|^{\alpha}u)-\partial^{\alpha}(|v|^{\alpha}v)$ as a sum of the following factors
\begin{equation}
 \begin{aligned}
&\mathcal{D}_{0,1}:= (|u|^{\gamma}-|v|^{\gamma})\partial^{\alpha}u+|v|^{\gamma}(\partial^{\alpha}u-\partial^{\alpha}v),\\
&\mathcal{D}_{0,2}:=(|u|^{\gamma-2p}-|v|^{\gamma-2p})\partial^{l_1}(|u|^2)\dots\partial^{l_p}(|u|^2)\partial^{\alpha-\beta}u,\\
&\mathcal{D}_j:=|v|^{\gamma-2p}\partial^{l_1}(|v|^2)\dots \partial^{l_{j-1}}(|v|^2)\partial^{l_j}(|u|^2-|v|^2)\partial^{l_{j+1}}(|u|^2)\dots \partial^{l_p}(|u|^2)\partial^{\alpha-\beta}u,\\
&\mathcal{D}_{p+1}:=|v|^{\gamma-2p} \partial^{l_1}(|v|^2)\dots\partial^{l_p}(|v|^2)\partial^{\alpha-\beta}(u-v),
 \end{aligned}   
\end{equation}
where $0<\beta\leq \alpha$, $p=1,\dots,|\beta|$, $l_{1}+\dots+l_p=\beta$, $l_j\geq 1$, $j=1,\dots,p$. The estimate for $\mathcal{D}_{0,1}$, $\mathcal{D}_{0,2}$, $\mathcal{D}_{j}$ and $\mathcal{D}_{p+1}$ follows by similar arguments in \eqref{prodnonl}. We use \eqref{nonlinear1} and \eqref{differengammaest} to infer
\begin{equation*}
\begin{aligned}
   \|\mathcal{D}_{0,1}\|_{L^2}\leq &  c^{|\gamma-1|}|\gamma|(\eta^{|\gamma-1|}+\|u\|_{H^J}^{|\gamma-1|}+\|v\|_{H^J}^{|\gamma-1|})\|\partial^{\alpha}u\|_{L^2}\|u-v\|_{H^J}\\
   &+c^{|\gamma|}(\eta^{|\gamma|}+\|v\|_{H^J}^{|\gamma|})\|\partial^{\alpha}u-\partial^{\alpha}v\|_{L^2}\\
   \leq &  \Big(c^{|\gamma-1|}|\gamma|(\eta^{|\gamma-1|}+\|u\|_{H^J}^{|\gamma-1|}+\|v\|_{H^J}^{|\gamma-1|})(\|u\|_{H^J}+\|v\|_{H^J})\\
   &+c^{|\gamma|}(\eta^{|\gamma|}+\|u\|_{H^J}^{|\gamma|}+\|v\|_{H^J}^{|\gamma|})\Big)\|u-v\|_{H^J}.
\end{aligned}    
\end{equation*}
Using again \eqref{differengammaest}, we get
\begin{equation*}
\begin{aligned}
     \|\mathcal{D}_{0,2}\|_{L^2}\leq & c^{|\gamma-2p-1|}|\gamma-2p|\big(\eta^{|\gamma-2p-1|}+\|u\|_{H^J}^{|\gamma-2p-1|}+\|v\|_{H^J}^{|\gamma-2p-1|}\big)\\
     &\times \|\partial^{l_1}(|u|^2)\dots\partial^{l_p}(|u|^2)\partial^{\alpha-\beta}u\|_{L^2}\|u-v\|_{H^J}.
\end{aligned}   
\end{equation*}
The estimate of the $L^2$-norm of $\partial^{l_1}(|u|^2)\dots\partial^{l_p}(|u|^2)\partial^{\alpha-\beta}u$ follows by the same arguments in the estimate of $\|\mathcal{B}_1\|_{L^2}$ in \eqref{prodnonl} (see \eqref{firstsecondcases} and \eqref{thirdcase}), that is, we compare the derivative of the highest order with the number $J-\frac{N}{2}$ (note that by assumption $J>\frac{N}{2}$). In conclusion, we obtain 
\begin{equation*}
\begin{aligned}
     \|\mathcal{D}_{0,2}\|_{L^2}\leq & c^{|\gamma-2p-1|}(\widetilde{c})^{2p+1}|\gamma-2p|\big(\eta^{|\gamma-2p-1|}+\|u\|_{H^J}^{|\gamma-2p-1|}+\|v\|_{H^J}^{|\gamma-2p-1|}\big)\\
     &\times (\|u\|_{H^J}^{2p+1}+\|v\|_{H^J}^{2p+1})\|u-v\|_{H^J},
\end{aligned}   
\end{equation*}
for some constants $c, \widetilde{c}>0$ with the same dependence in \eqref{firstsecondcases} and \eqref{thirdcase}.
\smallskip

Now, let $j=1,\dots,p$. Applying \eqref{nonlinear1}, it follows
\begin{equation*}
\begin{aligned}
  \|\mathcal{D}_j\|_{L^2}+&\|\mathcal{D}_{p+1}\|_{L^2} \\
  \leq &\, c^{|\gamma-2p|}(\eta^{|\gamma-2p|}+\|u\|_{H^J}^{|\gamma-2p|}+\|v\|_{H^J}^{|\gamma-2p|})\\
  &\times\Big(\|\partial^{l_1}(|v|^2)\dots \partial^{l_{j-1}}(|v|^2)\partial^{l_j}(|u|^2-|v|^2)\partial^{l_{j+1}}(|u|^2)\dots \partial^{l_p}(|u|^2)\partial^{\alpha-\beta}u\|_{L^2}\\
  &\quad +\|\partial^{l_1}(|v|^2)\dots \partial^{l_p}(|v|^2)\partial^{\alpha-\beta}(u-v)\|_{L^2}\Big).
\end{aligned}    
\end{equation*}
By writing
\begin{equation*}
  \partial^{l_j}(|u|^2-|v|^2)= \partial^{l_j}((u-v)\overline{u})+\partial^{l_j}(v(\overline{u-v})),  
\end{equation*}
we argue as in \eqref{prodnonl}, using the ideas from \eqref{firstsecondcases} and \eqref{thirdcase}, to get
\begin{equation*}
\begin{aligned}
\|\mathcal{D}_j\|_{L^2}+\|\mathcal{D}_{p+1}\|_{L^2} \leq &c^{|\gamma-2p|}(\widetilde{c})^{2p+1}(\eta^{|\gamma-2p|}+\|u\|_{H^J}^{|\gamma-2p|}+\|v\|_{H^J}^{|\gamma-2p|})\\
  &\times(\|u\|_{H^J}^{2p}+\|v\|_{H^J}^{2p})\|u-v\|_{H^J}. 
\end{aligned}    
\end{equation*}
Collecting together the previous estimates, we complete the proof of \eqref{diffnonlinear}. 
\end{proof}


\subsection{Proof of Theorem \ref{mainTHM1}}
The proof of Theorem \ref{mainTHM1} follows similar strategy as in the deduction of Theorem \ref{mainTHM2}. We prove the existence and uniqueness of solutions for the Cauchy problem \eqref{NLS} with series nonlinearity, where we show the dependence of constants and convergence of the series determined by the nonlinearity. The continuous dependence is a consequence of our arguments in the proof of existence below, together with similar ideas to those given in the proof of Theorem \ref{mainTHM2}, so we omit that deduction.
\smallskip

\underline{Existence of Solutions}. We consider the integral formulation \eqref{IntegralEquation},  
where $\mathcal{N}(|u|)=\sum_{k=0}^{\infty}a_k|u|^{\gamma_k}$, and the coefficients $\{a_k\}$ satisfy \eqref{convergenceeq}. Let $u_0\in H^J(\mathbb{T}^N)$ and \eqref{nonvanish} for some $\eta>0$, and $J\in \mathbb{Z}^{+}$ with $J>\frac{N}{2}+s$. We consider the same complete space $\mathcal{Z}_{T,R}$ defined in \eqref{zspace} for some $T, R>0$ to be determined later. Since $\{e^{-it (-\Delta)^{\frac{s}{2}}}\}$ defines a group of isometries in $H^J(\mathbb{T}^N)$, we apply Lemma \ref{nonlinearestimates} to get
\begin{equation}\label{contraeq}
\begin{aligned}
\norm{\Psi(u)(t)}_{H^J}
&\le \norm{u_0}_{H^J} +\sum^\infty_{k=0}|a_k|\Big|\int_{0}^{t}\norm{|u(\tau)|^{\gamma_k}u(\tau)}_{H^J}\, d\tau \Big|\\
&\leq \|u_0\|_{H^J}+T\sum_{k=0}^{\infty} |a_k|\sup_{t\in[-T,T]}\Gamma_{1}(\gamma_k,2\eta,\|u(t)\|_{H^J})\\
&\leq \|u_0\|_{H^J}+T\sum_{k=0}^{\infty} |a_k|\,\Gamma_{1}(\gamma_k,2\eta,R),
\end{aligned}
\end{equation}
and
\begin{equation}\label{diffseries}
 \begin{aligned}
 \|\Psi(u)(t)-\Psi(v)(t)\|_{H^{J}}
 \leq &\sum_{k=0}^{\infty}|a_k|\Big|\int_{0}^{t}\||u(\tau)|^{\gamma_k}u(\tau)-|v(\tau)|^{\gamma_k}v(\tau)\|_{H^J}\, d\tau \Big|\\
 \leq & T\sum_{k=0}^{\infty}|a_k|\sup\limits_{t\in[-T,T]}\big(\Gamma_2(\gamma_k,2\eta,\|u(t)\|_{H^J},\|v(t)\|_{H^J})\|u(t)-v(t)\|_{H^J}\big)\\
 \leq & T\sum_{k=0}^{\infty}|a_k|\,\Gamma_2(\gamma_k,2\eta,R,R)\sup\limits_{t\in[-T,T]}\|u(t)-v(t)\|_{H^J},
 \end{aligned}
\end{equation}
where $\Gamma_1(\cdot,\cdot,\cdot)$ and  $\Gamma_2(\cdot,\cdot,\cdot,\cdot)$ are defined in \eqref{nonestimates} and \eqref{diffnonlinear}, respectively. 

Next, we check the infimum condition in \eqref{zspace} for $\Psi(u)$. This follows from the same strategy used in the proof  of Theorem \ref{mainTHM2}. We apply \eqref{nonvanish} and the linear estimates obtained from \eqref{integraldiff} and deduce
\begin{equation}\label{infeq}
\begin{aligned}
|\Psi(u)(t)|\geq |u_0|-c\,T\|u_0\|_{H^J}-c\,T\sum_{k=0}^{\infty} |a_k|\Gamma_{1}(\gamma_k,2\eta,R)\\
\geq \eta^{-1}-c\,T\|u_0\|_{H^J}-c\,T\sum_{k=0}^{\infty} |a_k|\Gamma_{1}(\gamma_k,2\eta,R).
\end{aligned}    
\end{equation}
Note that the condition \eqref{convergenceeq} ensures that 
\begin{equation}\label{condcoeff}
  \sum_{k=0}^{\infty} |a_k|\Gamma_{1}(\gamma_k,2\eta,R)<\infty,
\end{equation}
\begin{equation}
\sum_{k=0}^{\infty}|a_k|\Gamma_2(\gamma_k,2\eta,R,R)<\infty.
\end{equation}
Thus, let $R=2\norm{u_0}_{H^J}$. We take $T>0$ sufficiently small such that 

\begin{equation}\label{condcontrac2}
\left\{ \begin{aligned}
   &T\sum_{k=0}^{\infty} |a_k|\Gamma_{1}(\gamma_k,2\eta,R) \leq \frac{R}{2}, \\ &T\sum_{k=0}^{\infty}|a_k|\Gamma_2(\gamma_k,2\eta,R,R)\leq \frac{1}{2},\\
   &T\left( \frac{R}{2}+c\Gamma_1(2\eta,R)R \right)\leq \frac{1}{2}\eta^{-1}.
 \end{aligned}  \right. 
\end{equation}
As a consequence of the estimates \eqref{contraeq} and \eqref{infeq}, we find that for a finite time $T$, $\Psi(u)$ defines a mapping from $\mathcal Z_{T,R}$ to itself. From \eqref{diffseries} we conclude that $\Psi(u)$ is a contraction mapping. Collecting these estimates, for some time $T>0$ such that \eqref{condcontrac2} holds, $\Psi(u)$ yields a fixed-point solution. This completes the proof of existence.
\smallskip

\underline{Uniqueness of solutions}. Let $u, v\in C([-T,T];H^J(\mathbb{T}^N))$ be two solutions of \eqref{NLS} with the same initial condition, i.e., $u(0,x)=v(0,x)$ satisfying \eqref{uniquenessinfcond} for some $\eta>0$. We also consider $M>0$ fixed such that \eqref{Kbound} holds.  \\

Notice that $w=u-v$ solves the initial value problem \eqref{uniqueeq}. Thus, multiplying the equation in \eqref{uniqueeq} by $\overline{w}$, then integrating over $\mathbb{T}^N$, and taking the imaginary part of the resulting expression, we deduce
\begin{equation*}
\begin{aligned}
 \frac{1}{2}\frac{d}{dt}\int_{\mathbb{T}^N} |w|^2\, dx=-\Im \int_{\mathbb{T}^N}\big(\mathcal{N}(|u|)u-\mathcal{N}(|v|)v\big)\overline{w}\, dx \\
   \leq \sum_{k=0}^{\infty}|a_k|\int_{\mathbb{T}^N}\big||u|^{\gamma_k}u-|v|^{\gamma_k}v\big||w|\, dx.   
\end{aligned}
\end{equation*}
To estimate the right-hand side of the expression above, we write $|u|^{\gamma_k}u-|v|^{\gamma_k}v=|u|^{\gamma_k}w+(|u|^{\gamma_k}-|v|^{\gamma_k})v$ to find
\begin{equation*}
 \begin{aligned}
 \int_{\mathbb{T}^N}\big||u|^{\gamma_k}u-|v|^{\gamma_k}v\big||w|\, dx\leq  \int_{\mathbb{T}^N}|u|^{\gamma_k}|w|^2\, dx+ \int_{\mathbb{T}^N}\big||u|^{\gamma_k}-|v|^{\gamma_k}\big||v||w|\, dx.    
 \end{aligned}   
\end{equation*}
Let us estimate the right-hand side of the above inequality. Given the validity of \eqref{uniquenessinfcond}, we can use \eqref{nonlinear1} to get
\begin{equation*}
\begin{aligned}
\int_{\mathbb{T}^N}|u|^{\gamma_k}|w|^2\, dx\leq & c^{|\gamma_k|}\big(\eta^{|\gamma_k|}+\|u\|_{H^J}^{|\gamma_k|}\big)\int_{\mathbb{T}^N} |w|^2\, dx\\
\leq & c^{|\gamma_k|}\big(\eta^{|\gamma_k|}+M^{|\gamma_k|}\big)\int_{\mathbb{T}^N} |w|^2\, dx,
\end{aligned}    
\end{equation*}
and we use \eqref{differengammaest} to infer
\begin{equation*}
 \begin{aligned}
 \int_{\mathbb{T}^N} &\big||u|^{\gamma_k}-|v|^{\gamma_k}\big||v||w|\, dx\\
 &\leq c^{|\gamma_k-1|+1}|\gamma_k|\big(\eta^{|\gamma_k-1|}+\|u\|_{H^J}^{|\gamma_k-1|}+\|v\|_{H^J}^{|\gamma_k-1|}\big)  \|v\|_{H^J}\int_{\mathbb{T}^N} |w|^2\, dx \\
 &\leq c^{|\gamma_k-1|+1}|\gamma_k|\big(\eta^{|\gamma_k-1|}+2M^{|\gamma_k-1|}\big) M\int_{\mathbb{T}^N} |w|^2\, dx.  
 \end{aligned}   
\end{equation*}
We remark that the above estimates hold for some universal constant $c>0$. Collecting the previous estimates, we arrive at
\begin{equation}\label{diffeesti}
  \frac{1}{2}\frac{d}{dt}\int_{\mathbb{T}^N}|w|^2\, dx\leq \Gamma_3(\eta,M)\int_{\mathbb{T}^N}|w|^2\, dx,  
\end{equation}
where
\begin{equation*}
 \begin{aligned}
\Gamma_3(\eta,M)\stackrel{def}{=}\sum_{k=0}^{\infty}|a_k|\Big(c^{|\gamma_k|}(\eta^{|\gamma_k|}+M^{|\gamma_k|})+ c^{|\gamma_k-1|+1}|\gamma_k|\big(\eta^{|\gamma_k-1|}+2M^{|\gamma_k-1|}\big) M\Big)<\infty,  
 \end{aligned}   
\end{equation*}
and the above expression is finite by condition \eqref{convergenceeq}. At this point, Gronwall's inequality and \eqref{diffeesti} yield the desired uniqueness property $w=0$, i.e., $u=v$, which concludes the proof of Theorem \ref{mainTHM1}.


\bibliographystyle{acm} 
\bibliography{references}

\begin{thebibliography}{10}

\bibitem{AroraRianoRoudenko2022}
{\sc Arora, A.~K., Ria\~{n}o, O., and Roudenko, S.}
\newblock Well-posedness in weighted spaces for the generalized {H}artree
  equation with {$p<2$}.
\newblock {\em Commun. Contemp. Math. 24}, 9 (2022), Paper No. 2150074.

\bibitem{BFG2023}
{\sc Bellazzini, J., Forcella, L., and Georgiev, V.}
\newblock Ground state energy threshold and blow-up for {NLS} with competing
  nonlinearities.
\newblock {\em Ann. Sc. Norm. Super. Pisa Cl. Sci. (5) 24}, 2 (2023), 955--988.

\bibitem{ArpadTadahiro2013}
{\sc B\'{e}nyi, A., and Oh, T.}
\newblock The {S}obolev inequality on the torus revisited.
\newblock {\em Publ. Math. Debrecen 83}, 3 (2013), 359--374.

\bibitem{B1998}
{\sc Berg\'e, L.}
\newblock Wave collapse in physics: principles and applications to light and
  plasma waves.
\newblock {\em Phys. Rep. 303}, 5-6 (1998), 259--370.

\bibitem{bm1976}
{\sc Bia{\l}ynicki-Birula, I., and Mycielski, J.}
\newblock Nonlinear wave mechanics.
\newblock {\em Ann. Physics 100}, 1-2 (1976), 62--93.

\bibitem{B1993}
{\sc Bourgain, J.}
\newblock Fourier transform restriction phenomena for certain lattice subsets
  and applications to nonlinear evolution equations. {I}. {S}chr\"odinger
  equations.
\newblock {\em Geom. Funct. Anal. 3}, 2 (1993), 107--156.

\bibitem{Carles2018}
{\sc Carles, R., and Gallagher, I.}
\newblock Universal dynamics for the defocusing logarithmic {S}chr\"odinger
  equation.
\newblock {\em Duke Math. J. 167}, 9 (2018), 1761--1801.

\bibitem{CS2021}
{\sc Carles, R., and Sparber, C.}
\newblock Orbital stability vs. scattering in the cubic-quintic
  {S}chr\"{o}dinger equation.
\newblock {\em Rev. Math. Phys. 33}, 3 (2021), Paper No. 2150004, 27.

\bibitem{Caz1983}
{\sc Cazenave, T.}
\newblock Stable solutions of the logarithmic {S}chr\"odinger equation.
\newblock {\em Nonlinear Anal. 7}, 10 (1983), 1127--1140.

\bibitem{CazHauNaum2020}
{\sc Cazenave, T., Han, Z., and Naumkin, I.}
\newblock Asymptotic behavior for a dissipative nonlinear {S}chr\"{o}dinger
  equation.
\newblock {\em Nonlinear Anal. 205\/} (2021), Paper No. 112243.

\bibitem{CH1980}
{\sc Cazenave, T., and Haraux, A.}
\newblock {\'E}quations d'\'evolution avec non lin\'earit\'e{} logarithmique.
\newblock {\em Ann. Fac. Sci. Toulouse Math. (5) 2}, 1 (1980), 21--51.

\bibitem{CazNaum2016}
{\sc Cazenave, T., and Naumkin, I.}
\newblock {Local existence, global existence, and scattering for the nonlinear
  {S}chr\"odinger equation}.
\newblock {\em Comm. Contemp. Math. 19}, 02 (2017), 1650038, 20 pp.

\bibitem{CazNaum2018}
{\sc Cazenave, T., and Naumkin, I.}
\newblock Modified scattering for the critical nonlinear {S}chr\"{o}dinger
  equation.
\newblock {\em J. Funct. Anal. 274}, 2 (2018), 402--432.

\bibitem{CHKL2015}
{\sc Cho, Y., Hwang, G., Kwon, S., and Lee, S.}
\newblock Well-posedness and ill-posedness for the cubic fractional
  {S}chr\"{o}dinger equations.
\newblock {\em Discrete Contin. Dyn. Syst. 35}, 7 (2015), 2863--2880.

\bibitem{dAMS2014}
{\sc d'Avenia, P., Montefusco, E., and Squassina, M.}
\newblock On the logarithmic {S}chr\"odinger equation.
\newblock {\em Commun. Contemp. Math. 16}, 2 (2014), 1350032, 15.

\bibitem{DET2016}
{\sc Demirbas, S., Erdo\u{g}an, M.~B., and Tzirakis, N.}
\newblock Existence and uniqueness theory for the fractional {S}chr\"{o}dinger
  equation on the torus.
\newblock In {\em Some topics in harmonic analysis and applications}, vol.~34
  of {\em Adv. Lect. Math. (ALM)}. Int. Press, Somerville, MA, 2016,
  pp.~145--162.

\bibitem{EGT2019}
{\sc Erdo\u{g}an, M.~B., G\"urel, T.~B., and Tzirakis, N.}
\newblock Smoothing for the fractional {S}chr\"odinger equation on the torus
  and the real line.
\newblock {\em Indiana Univ. Math. J. 68}, 2 (2019), 369--392.

\bibitem{EZ2008}
{\sc Erdo\u{g}an, M.~B., and Zharnitsky, V.}
\newblock Quasi-linear dynamics in nonlinear {S}chr\"odinger equation with
  periodic boundary conditions.
\newblock {\em Comm. Math. Phys. 281}, 3 (2008), 655--673.

\bibitem{FP2000}
{\sc Fibich, G., and Papanicolaou, G.}
\newblock Self-focusing in the perturbed and unperturbed nonlinear
  {S}chr\"odinger equation in critical dimension.
\newblock {\em SIAM J. Appl. Math. 60}, 1 (2000), 183--240.

\bibitem{FRRSY}
{\sc Friedman, I., Ria\~{n}o, O., Roudenko, S., Son, D., and Yang, K.}
\newblock {W}ell-posedness and dynamics of solutions to the generalized {K}d{V}
  with low power nonlinearity.
\newblock {\em Nonlinearity 36}, 1 (2023), 584--635.

\bibitem{PhysRevE1997}
{\sc Gaididei, Y.~B., Mingaleev, S.~F., Christiansen, P.~L., and Rasmussen,
  K.~O.}
\newblock Effects of nonlocal dispersive interactions on self-trapping
  excitations.
\newblock {\em Phys. Rev. E 55\/} (May 1997), 6141--6150.

\bibitem{Gi1996}
{\sc Ginibre, J.}
\newblock Le probl\`eme de {C}auchy pour des {EDP} semi-lin\'eaires
  p\'eriodiques en variables d'espace (d'apr\`es {B}ourgain).
\newblock No.~237. 1996, pp.~Exp. No. 796, 4, 163--187.
\newblock S\'eminaire Bourbaki, Vol.\ 1994/95.

\bibitem{GuoBolingYong2008}
{\sc Guo, B., Han, Y., and Xin, J.}
\newblock Existence of the global smooth solution to the period boundary value
  problem of fractional nonlinear {S}chr\"{o}dinger equation.
\newblock {\em Appl. Math. Comput. 204}, 1 (2008), 468--477.

\bibitem{HS2021}
{\sc Hakkaev, S., and Stefanov, A.~G.}
\newblock Stability of periodic waves for the fractional {K}d{V} and {NLS}
  equations.
\newblock {\em Proc. Roy. Soc. Edinburgh Sect. A 151}, 4 (2021), 1171--1203.

\bibitem{IP2014}
{\sc Ionescu, A.~D., and Pusateri, F.}
\newblock Nonlinear fractional {S}chr\"{o}dinger equations in one dimension.
\newblock {\em J. Funct. Anal. 266}, 1 (2014), 139--176.

\bibitem{KenigPonceVega1996}
{\sc Kenig, C.~E., Ponce, G., and Vega, L.}
\newblock Quadratic forms for the {$1$}-{D} semilinear {S}chr\"odinger
  equation.
\newblock {\em Trans. Amer. Math. Soc. 348}, 8 (1996), 3323--3353.

\bibitem{KLS2013}
{\sc Kirkpatrick, K., Lenzmann, E., and Staffilani, G.}
\newblock On the continuum limit for discrete {NLS} with long-range lattice
  interactions.
\newblock {\em Comm. Math. Phys. 317}, 3 (2013), 563--591.

\bibitem{LamLippmanTappert1977}
{\sc Lam, J.~F., Lippman, B., and Tappert, F.}
\newblock Self-trapped laser beams in plasma.
\newblock {\em Phys. Fluid 20}, 7 (1977), 1176--1179.

\bibitem{Laskin2000}
{\sc Laskin, N.}
\newblock Fractional quantum mechanics and {L}\'{e}vy path integrals.
\newblock {\em Phys. Lett. A 268}, 4-6 (2000), 298--305.

\bibitem{Laskin2002}
{\sc Laskin, N.}
\newblock Fractional {S}chr\"{o}dinger equation.
\newblock {\em Phys. Rev. E (3) 66}, 5 (2002), 056108.

\bibitem{LinaresMiyazakiPonce2019}
{\sc Linares, F., Miyazaki, H., and Ponce, G.}
\newblock On a class of solutions to the generalized {K}d{V} type equation.
\newblock {\em Commun. Contemp. Math. 21}, 7 (2019), 1850056.

\bibitem{LinaresPonce2015}
{\sc Linares, F., and Ponce, G.}
\newblock {\em Introduction to {N}onlinear {D}ispersive {E}quations},
  {S}econd~ed.
\newblock Universitext. Springer, New York, 2015.

\bibitem{LinaresPonceGleison2019}
{\sc Linares, F., Ponce, G., and Santos, G.~N.}
\newblock On a class of solutions to the generalized derivative
  {S}chr\"{o}dinger equations.
\newblock {\em Acta Math. Sin. (Engl. Ser.) 35}, 6 (2019), 1057--1073.

\bibitem{LinaresPonceGleison2019II}
{\sc Linares, F., Ponce, G., and Santos, G.~N.}
\newblock On a class of solutions to the generalized derivative
  {S}chr\"{o}dinger equations {II}.
\newblock {\em J. Differential Equations 267}, 1 (2019), 97--118.

\bibitem{MA1981}
{\sc Ma, Y.~C., and Ablowitz, M.~J.}
\newblock The periodic cubic {S}chr\"odinger equation.
\newblock {\em Stud. Appl. Math. 65}, 2 (1981), 113--158.

\bibitem{M2019}
{\sc Malomed, B.~A.}
\newblock {V}ortex solitons: {O}ld results and new perspectives.
\newblock {\em Phys. D 399\/} (2019), 108--137.

\bibitem{DNA1999}
{\sc Mingaleev, S., Christiansen, P., Gaididei, Y., Johannson, M., and
  Rasmussen, K.}
\newblock Models for energy and charge transport and storage in biomolecules.
\newblock {\em J. Biol. Phys. 25\/} (1999), 41--63.

\bibitem{Miyazaki2020}
{\sc Miyazaki, H.}
\newblock Lower bound for the lifespan of solutions to the generalized {K}d{V}
  equation with low degree of nonlinearity.
\newblock {\em Adv. Stud. Pure Math. 85\/} (2020), 303--313.

\bibitem{M2009}
{\sc Molinet, L.}
\newblock On ill-posedness for the one-dimensional periodic cubic {S}chrodinger
  equation.
\newblock {\em Math. Res. Lett. 16}, 1 (2009), 111--120.

\bibitem{RRR}
{\sc Rodriguez, A.~D., Ria\~{n}o, O., and Roudenko, S.}
\newblock Nonlinear {S}chr\"odinger equation with combined nonlinearities.
\newblock {\em preprint\/}.

\bibitem{SRaymond1991}
{\sc Saint~Raymond, X.}
\newblock {\em Elementary introduction to the theory of pseudodifferential
  operators}.
\newblock Studies in Advanced Mathematics. CRC Press, Boca Raton, FL, 1991.

\bibitem{S1997}
{\sc Staffilani, G.}
\newblock On the growth of high {S}obolev norms of solutions for {K}d{V} and
  {S}chr\"odinger equations.
\newblock {\em Duke Math. J. 86}, 1 (1997), 109--142.

\bibitem{SulemSulem1999}
{\sc Sulem, C., and Sulem, P.-L.}
\newblock {\em The nonlinear {S}chr\"odinger equation: Self-focusing and wave
  collapse}, vol.~139 of {\em Applied Mathematical Sciences}.
\newblock Springer-Verlag, New York, 1999.

\bibitem{Tao2006}
{\sc Tao, T.}
\newblock {\em Nonlinear dispersive equations: Local and global analysis},
  vol.~106 of {\em CBMS Regional Conference Series in Mathematics}.
\newblock American Mathematical Society, Providence, RI, 2006.

\bibitem{TVZ2007}
{\sc Tao, T., Visan, M., and Zhang, X.}
\newblock The nonlinear {S}chr\"odinger equation with combined power-type
  nonlinearities.
\newblock {\em Comm. Partial Differential Equations 32}, 7-9 (2007),
  1281--1343.

\bibitem{Thirouin2017}
{\sc Thirouin, J.}
\newblock On the growth of {S}obolev norms of solutions of the fractional
  defocusing {NLS} equation on the circle.
\newblock {\em Ann. Inst. H. Poincar\'{e} C Anal. Non Lin\'{e}aire 34}, 2
  (2017), 509--531.

\bibitem{Z1967}
{\sc Zakharov, V.}
\newblock Instability of self-focusing of light.
\newblock {\em Zh. Eksp. Teor. Fiz. 53\/} (1967), 1735--1743 [Sov. Phys. JETP,
  26, (1968), no. 5, 994--998].

\end{thebibliography}
\end{document}